\documentclass[12pt,a4paper,twoside,reqno]{amsart}
\usepackage{tikz}
\usepackage{slashed}
\usepackage{amsmath,amscd}
\usepackage{amssymb,amsfonts,mathrsfs}
\usepackage[driver=pdftex,margin=3cm,heightrounded=true,centering]{geometry}
\usepackage{JK}
\usepackage[colorlinks=true,linkcolor=blue,citecolor=blue]{hyperref}
\newtheorem{thm}{Theorem}[section]

\usepackage{graphicx}

\begin{document}
\title[The unbounded Kasparov product by a differentiable module]{The unbounded Kasparov product by a differentiable module}

\author{Jens Kaad}

\address{Department of Mathematics and Computer Science,
The University of Southern Denmark,
Campusvej 55, DK-5230 Odense M,
Denmark}

\email{kaad@imada.sdu.dk}

%\thanks{The first author thanks the Alexander von Humboldt Stiftung and colleagues at the University of M\"unster and acknowledges the support of the Australian Research Council. The second author was partly supported by the Fondation Sciences Math\'ematiques de Paris. Both authors are very appreciative of the support offered by the Erwin Schr\"odinger Institute where much of this research was carried out. We are also grateful for the advice of Joachim Cuntz, Harald Grosse and Fritz Gesztesy while this investigation was proceeding.}

%\thanks{2010 \emph{Mathematical Subject Classification}: Primary: 58B32, %Secondary: 58B34, 33D80, 19D55, 81R50.}
%\thanks{The first author is supported by the Fondation Sciences Mat\'ematiques de Paris.}
%\thanks{The second author is partially supported by the Danish National Research Foundation (DNRF) through the Centre for Symmetry and Deformation.}

\begin{abstract}%{{{
In this paper we investigate the unbounded Kasparov product between a differentiable module and an unbounded cycle of a very general kind that includes all unbounded Kasparov modules and hence also all spectral triples. Our assumptions on the differentiable module are weak and we do in particular not require that it satisfies any kind of smooth projectivity conditions. The algebras that we work with are furthermore not required to possess a smooth approximate identity. The lack of an adequate projectivity condition on our differentiable module entails that the usual class of unbounded Kasparov modules is not flexible enough to accommodate the unbounded Kasparov product and it becomes necessary to twist the commutator condition by an automorphism.

We show that the unbounded Kasparov product makes sense in this twisted setting and that it recovers the usual interior Kasparov product after taking bounded transforms. Since our unbounded cycles are twisted (or modular) we are not able to apply the work of Kucerovsky for recognizing unbounded representatives for the bounded Kasparov product, instead we rely directly on the connection criterion developed by Connes and Skandalis. In fact, since we do not impose any twisted Lipschitz regularity conditions on our unbounded cycles, even the passage from an unbounded cycle to a bounded Kasparov module requires a substantial amount of extra care.
\end{abstract}%}}}

\maketitle
\tableofcontents
%\listoffigures
\section{Introduction}
In a series of papers from the early eighties, Kasparov proved the fundamental results on the $KK$-theory of $C^*$-algebras, \cite{Kas:HSV, Kas:OFE, Kas:TIE}. One of the main inventions appearing in these papers is the interior Kasparov product which provides a bilinear and associative pairing
\[
\hot_B : KK_n(A,B) \ti KK_m(B,C) \to KK_{n+m}(A,C)
\]
between the $KK$-groups of three separable $C^*$-algebras $A,B$ and $C$. The interior Kasparov product of two $KK$-classes is computable in many cases, but the main construction remains inexplicit as it relies on Kasparov's absorption theorem and Kasparov's technical theorem.

One of the advantages of the $KK$-groups of $C^*$-algebras is the wealth of explicit examples of elements arising from geometric data. Indeed, in the unbounded picture of $KK$-theory the cycles are unbounded Kasparov modules, which are bivariant versions of Connes' concept of a spectral triple, and these unbounded Kasparov modules exhaust the $KK$-groups as proved by Baaj and Julg, \cite{BaJu:TBK}.

The problem that we are concerned with in this paper is to construct an unbounded version of the interior Kasparov product. More precisely, starting with two unbounded Kasparov modules, the aim is to find an \emph{explicit} unbounded Kasparov module that represents the interior Kasparov product. In particular, this construction should bypass the need for invoking both the absorption theorem and the technical theorem. The problem of constructing the unbounded Kasparov product is currently receiving an increasing amount of attention, see \cite{Con:GCM, KaLe:SFU, Mes:UCN, MeRe:NST}, as is also witnessed by the quantity of recent applications, see for example \cite{BrMeSu:GSU, MeGo:STF, ReFo:FES, BoCaRe:BCQ, GoMeRe:SEU, Dun:IDS, KaSu:RSF, KaSu:FDT}.

More specifically, the techniques appearing in the present paper have already been applied to the study of Morita equivalences of spectral triples, unbounded Kasparov products in the context of Hilsum's half-closed chains, and to factorization problems for Dirac operators along the orbits of proper but not necessarily free actions, \cite{Kaa:MEU, KaSu:KHC, KaSu:FDA, KaSu:FDT}. In fact, in these applications, the techniques developed here constitute an essential ingredient.  

At a deeper level, the unbounded Kasparov product is important because of the loss of geometric information that is inherent in the passage from an unbounded Kasparov module to a class in $KK$-theory. It is thus in our interest to develop a version of the interior Kasparov product retaining a larger amount of geometric data (relating to the asymptotic behaviour of eigenvalues of differential operators).

In this paper we are focusing on the case where the class in the $KK$-group, $KK(A,B)$, is represented by a $C^*$-correspondence $X$ from $A$ to $B$ and where the action of $A$ from the left factorizes through the $C^*$-algebra of compact operators on $X$. On the other hand, our class in the $KK$-group $KK(B,C)$ will be represented by an unbounded selfadjoint and regular operator $D : \sD(D) \to Y$ acting on a $C^*$-correspondence from $B$ to $C$. The unbounded operator $D$ is required to satisfy a couple of extra conditions that will be detailed out in the main text. The first challenge is then to construct a new unbounded selfadjoint and regular operator
\[
1 \ot_{\Na} D : \sD(1 \ot_{\Na} D) \to X \hot_B Y
\]
that acts on the interior tensor product of the $C^*$-correspondences $X$ and $Y$. In the main part of the earlier work on the unbounded Kasparov product this step is accomplished by assuming the existence of a (tight normalized) frame $\{ \ze_k\}$ for $X$ (see \cite{FrLa:FHM}) such that the associated orthogonal projection
\[
P := \sum_{n,m = 1}^\infty \inn{\ze_n, \ze_m} \de_{nm} : \ell^2(\nn,Y) \to \ell^2(\nn,Y) 
\]
(which acts on the standard module over $Y$) has a bounded commutator with the unbounded selfadjoint and regular operator $D : \sD(D) \to Y$ (weaker but related conditions are applied in \cite{BrMeSu:GSU} and \cite{MeRe:NST}). The unbounded selfadjoint and regular operator $1 \ot_{\Na} D : \sD(1 \ot_{\Na} D) \to X \hot_B Y$ can then be expressed as the infinite sum
\[
1 \ot_{\Na} D := \ov{ \sum_{n = 1}^\infty T_{\ze_n} D T_{\ze_n}^* } ,
\]
where $T_{\ze_n} : Y \to X \hot_B Y$, $y \mapsto \ze_n \ot_B y$, is the creation operator associated with the element $\ze_n \in X$. It should be noted that the unbounded selfadjoint and regular operator $1 \ot_{\Na} D$ can be described in an alternative way by using the notion of a densely defined covariant derivative $\Na$ on the $C^*$-correspondence $X$. Indeed, the frame $\{ \ze_k\}$ gives rise to a Grassmann covariant derivative $\Na_{\T{Gr}}$ and the unbounded selfadjoint and regular operator $1 \ot_{\Na} D$ is then given by the (closure of the) sum $c(\Na_{\T{Gr}}) + 1 \ot D$ where the ``$c$'' refers to an appropriate notion of Clifford multiplication.

One of the main contributions of this paper is that we have been able to entirely remove the above smooth projectivity condition on the $C^*$-correspondence $X$. This step is motivated by the detailed investigations of differentiable structures in Hilbert $C^*$-modules carried out in \cite{Kaa:DAH,Kaa:SSB}. In particular, we find that the removal of smooth projectivity is relevant for accommodating examples arising from non-complete manifolds. 

Instead of imposing a smooth projectivity condition we simply assume that there exists a sequence of generators $\{ \xi_k\}$ for $X$ such that the associated operator
\[
G := \sum_{n,m = 1}^\infty \inn{\xi_n, \xi_m} \de_{nm} : \ell^2(\nn,Y) \to \ell^2(\nn,Y) 
\]
has a bounded commutator with (the diagonal operator induced by) $D : \sD(D) \to Y$. We then obtain a new unbounded selfadjoint and regular operator
\[
D_\De := \ov{ \sum_{n = 1}^\infty T_{\xi_n} D T_{\xi_n}^* }
\]
on the interior tensor product $X \hot_B Y$. We refer to this unbounded selfadjoint and regular operator as the \emph{modular lift} of $D : \sD(D) \to Y$. The fact that our sequence $\{ \xi_k\}$ is no longer assumed to be a frame means that we obtain an extra (non-trivial) bounded adjointable operator
\[
\De := \sum_{n = 1}^\infty T_{\xi_n} T_{\xi_n}^* : X \hot_B Y \to X \hot_B Y
\]
on the interior tensor product. An investigation of the commutators between the algebra elements in $A$ and the modular lift now shows that the usual straight commutator has to be replaced by a twisted commutator, where the twist is given by the (modular) automorphism $\si$ obtained from conjugation with the modular operator $\De$. This modular automorphism corresponds to the analytic extension at $-i \in \cc$ of the modular group of automorphisms $\si_t : T \mapsto \De^{it} T \De^{-it}$, $t \in \rr$. The notion of an unbounded modular cycle appearing in this text is thus related to the concept of a twisted spectral triple as introduced by Connes and Moscovici in \cite{CoMo:TST}. We remark however that our modular automorphism $\si$ is not required to be densely defined on the algebra $A$ and it need not preserve this algebra either. The concept of an unbounded modular cycle is a more flexible concept than the more commonly encountered notion of an unbounded Kasparov module. The extra flexibility comes from the presence of the modular operator $\De$ and in fact unbounded Kasparov modules are exactly unbounded modular cycles, where the modular operator $\De$ equals the identity operator.
%
%This extra ``modular'' operator $\De$ is positive and has dense image but it will only equal the identity when $\{ \xi_k \}$ is a frame. 

Our first main result can now be stated as follows, where we refer to the main text for the precise definitions:

\begin{thm}\label{t:unbkasint}
Suppose that $X$ is a differentiable $C^*$-correspondence with left action factorizing through the compacts and that $(Y,D,\Ga)$ is an unbounded modular cycle (with modular operator $\Ga : Y \to Y$). Then the triple
\[
(X \hot_B Y, D_\De, \De)
\]
is an unbounded modular cycle, where the new modular operator is defined by $\De := \sum_{n = 1}^\infty T_{\xi_n} \Ga T_{\xi_n}^*$.
\end{thm} 

The second central theme of this paper develops around the relationship between the assignment
\[
\big( X, (Y,D,\Ga) \big) \mapsto \big(X \hot_B Y, D_\De, \De \big)
\]
and the interior Kasparov product $KK_0(A,B) \ti KK_*(B,C) \to KK_*(A,C)$. As a first step, we have to understand how to produce a class in $KK$-theory from an unbounded modular cycle. We announce the following theorem:

\begin{thm}\label{t:bddkasint}
Suppose that $(Y,D,\Ga)$ is an unbounded modular cycle (relating the $C^*$-algebras $B$ and $C$). Then the pair $(Y, D(1 + D^2)^{-1/2})$ is a bounded Kasparov module from $B$ to $C$ and we thus have a class $[D] \in KK_*(B,C)$.
\end{thm}

Of course this theorem is a direct analogue of the theorem of Baaj and Julg showing how to construct a class in $KK$-theory from an unbounded Kasparov module. The proof of this result in the context of unbounded modular cycle is however far more involved. One reason for this extra difficulty can be found in the seemingly innocent change from straight commutators to twisted commutators. Indeed, an examination of the proof appearing in \cite{BaJu:TBK} shows that the crucial step fails for algebraic reasons when applied to unbounded modular cycles. An alternative approach would be to follow Connes and Moscovici and replace $(1 + D^2)^{-1/2}$ by $(1 + |D|)^{-1}$, see \cite{CoMo:TST}. This alternative approach does however rely on an extra assumption of twisted Lipschitz regularity and we do not impose this kind of extra regularity conditions on our unbounded modular cycle. Indeed, it is unclear how twisted Lipschitz regularity behaves with respect to the unbounded Kasparov product given in Theorem \ref{t:unbkasint}. In order to prove Theorem \ref{t:bddkasint} we therefore found it necessary to develop new techniques that can be applied to non-Lipschitz unbounded modular cycles.

The main new tool appearing in the proof of Theorem \ref{t:bddkasint} is the \emph{modular transform} $G_{D,\Ga} : \Ga( \sD(D)) \to Y$ which is given by the (absolutely convergent) integral
\[
G_{D,\Ga} : \Ga \xi \mapsto \frac{1}{\pi} \int_0^\infty \la^{-1/2} \Ga (1 + \la \Ga^2 + D^2)^{-1} D( \Ga \xi) \, d\la
\]
for all $\xi \in \sD(D)$. The modular transform is obtained from the usual bounded transform by making a non-commutative change of variables corresponding to $\mu := \la \Ga^2$. This change of variables is motivated by the observation that the modular transform (contrary to the bounded transform) has the right commutator properties with elements in the algebra $B$. A substantial part of the proof of Theorem \ref{t:bddkasint} is then devoted to a comparison between the bounded transform and the modular transform. Notice that the modular transform does not in general seem to have a bounded extension to $Y$ but that a sufficient condition for this to happen is that the modular operator $\Ga : Y \to Y$ has a bounded inverse.

With the knowledge of the relationship between unbounded modular cycles and classes in $KK$-theory in place, we can state our second main result:

\begin{thm}\label{t:relkasint}
Suppose that $X$ is a differentiable $C^*$-correspondence with left action factorizing through the compacts and suppose that $(Y,D,\Ga)$ is an unbounded modular cycle. Let $[X] \in KK_0(A,B)$ and $[D] \in KK_*(B,C)$ denote the corresponding classes in $KK$-theory. Let also $[D_\De] \in KK_*(A,C)$ denote the $KK$-class of the unbounded modular cycle $(X \hot_B Y, D_\De, \De)$. Then we have the identity
\[
[D_\De] = [X] \hot_B [D]
\]
in the $KK$-group $KK_*(A,C)$.
\end{thm}

The proof of this theorem does not follow the usual scheme in unbounded $KK$-theory. Indeed, the standard method that is available for recognizing an unbounded representative for the interior Kasparov product is to use the machinery invented by Kucerovsky, \cite{Kuc:PUM,KUC:LIK}. However, the results of Kucerovsky do not apply in the context of unbounded modular cycles because of our systematic use of twisted commutators instead of straight commutators. Rather than applying Kucerovsky's ideas we rely directly on the notion of an $F_2$-connection as introduced by Connes and Skandalis, \cite{CoSk:LIF}.

Let us end this introduction by giving a more tangible corollary to our main theorems. Consider a countable union $U := \cup_{k = 1}^\infty I_k$ of bounded open intervals $I_k \su \rr$. For each $k \in \nn$ we then choose a smooth function $f_k : \rr \to \rr$ with support equal to the closure $\ov{I_k} \su \rr$. After a rescaling we may assume that $\| f_k \| + \| \frac{df_k}{dx} \| \leq 1/k$ for all $k \in \nn$ (where $\|\cd \|$ denotes the supremum norm). Define the first order differential operator
\[
(D_\De)_0 := i \sum_{k = 1}^\infty f_k^2 \frac{d}{dx} 
+ i \sum_{k = 1}^\infty f_k \frac{df_k}{dx} : C_c^\infty(U) \to L^2(U)
\]
and let $D_\De := \ov{ (D_\De)_0}$ denote the closure. We then have the following result:

\begin{cor}
The triple $(C_c^\infty(U), L^2(U), D_\De)$ is an odd spectral triple and the associated class in the odd $K$-homology group $K^1(C_0(U))$ agrees with the interior Kasparov product of (the $KK$-classes associated with) the $C^*$-correspondence $C_0(U)$ from $C_0(U)$ to $C_0(\rr)$ and the Dirac operator on the real line.
\end{cor}

There is a similar kind of corollary, when the setting is given by an arbitrary spectral triple $(\sA,H,D)$ together with a sequence of elements $\{ x_k \}$ in the algebra $\sA$ such that $\| x_k \| + \| [D,x_k] \| \leq 1/k$ for all $k \in \nn$. When the algebra $\sA$ is non-commutative, it is however not true that one obtains a new spectral triple out of this construction. In the general case, it becomes necessary to twist all the commutators appearing by the modular operator $\De := \sum_{k = 1}^\infty x_k x_k^*$ and the framework we are developing here is therefore relevant for treating this kind of examples.

We announce one more corollary, which should be compared with the constructions in \cite{CoMo:TST}. %We emphasize that even in the case where $(\sA,H,D)$ is Lipschitz regular the perturbed spectral triple  

\begin{cor}
Let $(\sA,H,D)$ be a unital spectral triple (where the unit in $\sA$ acts as the identity operator on $H$) and let $g \in \sA$ be a positive and invertible element. Then the triple $(H, \ov{g D g}, g^2)$ is an unbounded modular cycle from $\sA$ to $\cc$. Moreover, the bounded transform $\big(H, \ov{g D g} (1 + \ov{g D g}^2)^{-1/2} \big)$ is a bounded Kasparov module from $A$ to $\cc$ and we have the identity
\[
[ H, \ov{g D g} (1 + \ov{g D g}^2)^{-1/2}] = [H, D(1 + D^2)^{-1/2}]
\]
in the $KK$-group $KK_*(A,\cc)$.
\end{cor}

In fact, the passage from the unital spectral triple $(\sA,H,D)$ to the unbounded modular cycle $(H, \ov{g D g}, g^2)$ is in this text interpreted as an unbounded Kasparov product with the differentiable $C^*$-correspondence $A$, where the differentiable structure is given by the single generator $g \in A$.

\subsection{Acknowledgement}
The union $U := \cup_{k = 1}^\infty I_k$ appearing in the introduction is referred to as a \emph{fractral string} when it is bounded and when the open intervals are disjoint. I am grateful to Michel Lapidus for making me aware of this example, \cite{LaFr:FGC}.

The research for this paper was initiated during the Hausdorff Trimester Program ``Non-commutative Geometry and its Applications'' at the Hausdorff Institute for Mathematics in Bonn and continued while the author was supported by the Radboud Excellence Initiative at the Radboud University in Nijmegen.

\section{Preliminaries on operator spaces}
We begin this paper by fixing our conventions for the analytic properties of the $*$-algebras appearing throughout this text. It turns out that the conventional setup of Banach spaces is not adequate for capturing the relevant structure on our $*$-algebras. Indeed, it will soon become apparent that one needs to fix the analytic behaviour not only of the $*$-algebra itself but of all the finite matrices with entries in the $*$-algebra. The notion of operator spaces is therefore providing the correct analytic setting and we will now briefly survey the main definitions. For more details we refer the reader to the books by Blecher-Le Merdy and by Pisier, \cite{BlMe:OMO,Pis:IOT}.

Let $H$ and $G$ be Hilbert spaces and let $X \su \sL(H,G)$ be a subspace (of the bounded operators from $H$ to $G$), which is closed in the operator norm. Then the vector space $M(X) := \lim_{n \to \infty} M_n(X)$ of finite matrices over $X$ has a canonical norm $\|\cd\|_X$ coming from the identifications $M_n(X) \su M_n(\sL(H,G)) \cong \sL(H^n,G^n)$. The properties of the pair $\big( M(X), \|\cd\|_X\big)$ are crystallized in the next definition.

Notice that the above construction yields a canonical norm $\|\cd\|_\cc : M(\cc) \to [0,\infty)$ on the finite matrices over $\cc$ since $\cc \cong \sL(\cc,\cc)$. For each $n \in \nn$, the norm $\|\cd\|_\cc : M_n(\cc) \su M(\cc) \to [0,\infty)$ coincides with the unique $C^*$-algebra norm.

\begin{dfn}\label{d:opespa}
An \emph{operator space} is a vector space $X$ over $\cc$ with a norm $\|\cd\|_X$ on the finite matrices $M(X) := \lim_{n \to \infty} M_n(X)$ such that
\begin{enumerate}
\item the normed space $X \su M(X)$ is a Banach space;
\item the inequality $\|v \cd \xi \cd w\|_X \leq \|v\|_{\cc} \cd \|\xi\|_X \cd \|w\|_\cc$ holds for all $v,w \in M(\cc)$ and all $\xi \in M(X)$;
\item the equality $\|\xi \op \eta \|_X = \T{max}\{\|\xi\|_X, \|\eta\|_X\}$ holds for all $\xi \in M_n(X)$ and $\eta \in M_m(X)$, where $\xi \op \eta \in M_{n+m}(X)$ is the direct sum of the matrices.
%
% projections $p,q \in M(\cc)$ with $pq = 0$ and all $\xi \in M(X)$.
%\item The equality $\|p \ot \xi\|_X = \|\xi\|_X$ holds for each projection $p \in M(\cc)$ of rank one and each $\xi \in M(X)$.
\end{enumerate}

A \emph{morphism} of operator spaces is a \emph{completely bounded} linear map $\al : X \to Y$. The term completely bounded means that $\al_n : M_n(X) \to M_n(Y)$ is a bounded operator for each $n \in \nn$ and that $\T{sup}_n \|\al_n\|_\infty < \infty$ (where $\| \cd \|_\infty$ is the operator norm). The supremum is denoted by $\|\al\|_{\T{cb}} := \T{sup}_n\|\al_n\|_\infty$ and is referred to as the \emph{completely bounded norm}.
\end{dfn}

By a fundamental theorem of Ruan every operator space $X$ is completely isometric to a closed subspace of $\sL(H)$ for some Hilbert space $H$, see \cite[Theorem 3.1]{Rua:SCA}.

We remark that any $C^*$-algebra $A$ carries a canonical operator space structure such that $M_n(A)$ becomes a $C^*$-algebra for all $n \in \nn$.

We will in this text mainly be concerned with dense subspaces of operator spaces. On such a dense subspace $\sX \su X$ we refer to the norm on $\sX$ coming from the surrounding operator space $X$ as an \emph{operator space norm} on $\sX$. We will often say that a linear map $\al : \sX \to Y$ with values in an operator space $Y$ is completely bounded when it extends to a completely bounded map $\al :X \to Y$. 

The next assumption will remain in effect throughout this paper:

\begin{assum}\label{a:alg}
Any $*$-algebra $\sA$ encountered in this text will come equipped with an operator space norm $\| \cd \|_1 : \sA \to [0,\infty)$ and a $C^*$-norm $\| \cd \| : \sA \to [0,\infty)$. We denote the operator space completion of $\sA$ by $A_1$ and the $C^*$-algebra completion by $A$. We assume that the inclusion $\sA \to A$ extends to a completely bounded and injective map $A_1 \to A$.
\end{assum}

We emphasize that we never assume the existence of an approximate identity $\{u_i\}_{i \in I}$ for the $C^*$-algebra $A$ with the additional property that $u_i \in \sA$ for all $i \in I$ and $\sup_{i \in I}\| u_i \|_1 < \infty$.

\subsubsection{Stabilization of operator spaces}\label{ss:stabspa}
Let us consider an operator space $X$. The following stabilization construction will play a central role in this paper. It does not make any sense when $X$ is merely a Banach space.

\begin{dfn}\label{d:staope}
By the \emph{stabilization} of $X$ we understand the operator space $K(X)$ obtained as the completion of the vector space of finite matrices $M(X)$ with respect to the canonical norm
\[
\| \cd \|_X : M(X) \to [0,\infty) .
\]
The matrix norms for $K(X)$ come from the matrix norms for $X$ via the canonical identification (forgetting the subdivisions):
\[
M_n(M_m(X)) \cong M_{n \cd m}(X) , \q n,m \in \nn .
\]
\end{dfn}

The above stabilization procedure is functorial: any completely bounded map $\al : X \to Y$ induces a completely bounded map $K(\al) : K(X) \to K(Y)$ by applying $\al$ entrywise and we have that $\| \al \|_{\T{cb}} = \| K(\al) \|_{\T{cb}}$.

%It is important to remark that when $A$ is a $C^*$-algebra we obtain that $\sK(A)$ is isomorphic (as a $C^*$-algebra) to the compact operators on the standard module $H_A$ over $A$ (see \cite{}).  

\section{Unbounded modular cycles}
Throughout this section we let $\sA$ be a $*$-algebra, which satisfies the conditions in Assumption \ref{a:alg}. We let $A_1$ denote the operator space completion of $\sA$ and $A$ denote the $C^*$-completion of $\sA$. We let $B$ be an arbitrary $C^*$-algebra.

Let us recall some basic constructions for a Hilbert $C^*$-module $X$ over $B$. For more details the reader may consult the book by Lance, \cite{Lan:HCM}.

The \emph{standard module} over $X$ is the Hilbert $C^*$-module $\ell^2(\nn,X)$ over $B$ consisting of all sequences $\sum_{n = 1}^\infty x_n \de_n$ in $X$ such that the sequence of partial sums $\big\{ \sum_{n = 1}^N \inn{x_n,x_n} \big\}$ converges in the norm on $B$. The right module structure is given by $\big( \sum_{n = 1}^\infty x_n \de_n\big) \cd b := \sum_{n = 1}^\infty (x_n \cd b) \de_n$ and the inner product is given by
\[
\binn{ \sum_{n = 1}^\infty x_n \de_n, \sum_{n = 1}^\infty y_n \de_n} := \sum_{n = 1}^\infty \inn{x_n,y_n} .
\]
%(where the convergence of the last sum follows from the Cauchy-Schwartz inequality).

The \emph{bounded adjointable operators} on $X$ is the $C^*$-algebra $\sL(X)$ consisting of all the bounded operators on $X$ that admit an adjoint with respect to the inner product on $X$. The $C^*$-norm on $\sL(X)$ is the operator norm $\| \cd \|_\infty$.
 
The \emph{compact operators} on $X$ is the $C^*$-algebra $\sK(X)$ defined as the operator norm closure of the $*$-subalgebra
\[
\sF(X) := \T{span}_\cc \{ \te_{\xi,\eta} \mid \xi, \eta \in X \} \su \sL(X) ,
\]
where $\te_{\xi,\eta} : X \to X$ is defined by $\te_{\xi,\eta}(\ze) := \xi \cd \inn{\eta,\ze}$ for all $\xi,\eta,\ze \in X$.

An unbounded densely defined operator $D : \sD(D) \to X$ is said to be \emph{symmetric} when $\inn{D \xi,\eta} = \inn{\xi,D \eta}$ for all $\xi,\eta \in \sD(D)$. An unbounded symmetric operator is \emph{selfadjoint} when the following implication holds for all $\eta \in X$: 
\[
\Big( \exists \ze \in X : \inn{D\xi,\eta} = \inn{\xi,\ze} \, \, \forall \xi \in \sD(D) \Big) \rar \big( \eta \in \sD(D) \big).
\]
An unbounded selfadjoint operator is \emph{regular} when the unbounded operators $D \pm i : \sD(D) \to X$ are surjective. %The dense submodule $\sD(D) \su X$ is referred to as the domain of $D$.

%For a bounded adjointable operator $T : X \to X$ we let $C^*(T) \su \sL(X)$ denote the $C^*$-subalgebra generated by $T$.
%
%Furthermore, for any operator norm-bounded sequence $\{T_n\}_{n = 1}^\infty$ of bounded adjointable operators on $X$ we define the bounded adjointable operator
%\[
%\T{diag}(T_n) : H_X \to H_X \q \sum_{n = 1}^\infty x_n \de_n \mapsto \sum_{n = 1}^\infty T_n(x_n) \de_n
%\]
\medskip

We are now ready to introduce the first of the main new concepts of the present paper:
\medskip

\begin{dfn}\label{d:unbkas}
An \emph{odd unbounded modular cycle} from $\sA$ to $B$ is a triple $(X,D,\De)$, where
\begin{enumerate}
\item $X$ is a countably generated Hilbert $C^*$-module over $B$, which comes equipped with a $*$-homomorphism $\pi : A \to \sL(X)$;
\item $D : \sD(D) \to X$ is an unbounded selfadjoint and regular operator on $X$;
\item $\De : X \to X$ is a bounded positive and selfadjoint operator with dense image,
\end{enumerate}
such that the following holds:
\begin{enumerate}
\item $\pi(a) \cd (i + D)^{-1} : X \to X$ is a compact operator for all $a \in A$;
\item $(\pi(a) + \la )  \De (\sD(D)) \su \sD(D)$ and 
\[
D (\pi(a) + \la) \De - \De (\pi(a) + \la) D : \sD(D) \to X
\]
extends to a bounded adjointable operator $d_\De(a + \la) : X \to X$ for all $a \in \sA$ and $\la \in \cc$;
\item the supremum %image of $d_\De(T)$ is contained in the image of $\De^{1/2}$ and the unbounded operator
\[
\sup_{\ep \in (0,1]}\| (\De + \ep)^{-1/2} d_\De(a + \la) (\De + \ep)^{-1/2} \|_\infty
\]
is finite for all $a \in \sA$ and $\la \in \cc$ and the linear map
\[
\rho_{\De,\ep} : \sA \to \sL(X) \q \rho_{\De,\ep}(a) = (\De + \ep)^{-1/2} d_{\De}(a) (\De + \ep)^{-1/2}
\]
extends to a completely bounded map $\rho_{\De,\ep} : A_1 \to \sL(X)$ for all $\ep \in (0,1]$ such that
\[
\sup_{\ep \in (0,1]}\| \rho_{\De,\ep} \|_{\T{cb}} < \infty .
\]
% (\De + \ep)^{-1/2} d_\De(a) (\De + \ep)^{-1/2} \|_\infty \leq C \cd \| a \|_1 
%\]
%for all $a \in M_n(\sA)$ and $n \in \nn$;
%\item The linear map $\rho_\De : \sA \to \sL(X)$ defined by
%\[
%a \mapsto \ov{ \De^{-1/2} d_\De(\pi(a)) \De^{-1/2} }
%\]
%is completely bounded;
\item the sequence $\{ \pi(a) \De (\De + 1/n)^{-1} \}_{n = 1}^\infty$ converges in operator norm to $\pi(a)$ for all $a \in A$.
%
%re exists a countable approximate identity $\{ V_n\}_{n = 1}^\infty$ for the $C^*$-algebra $C^*(\De)$ such that the sequence
%\[
%\big\{ V_n \pi(a) \big\}_{n=1}^\infty
%\]
%converges in operator norm to $\pi(a)$ for all $a \in A$.
\end{enumerate}
We refer to $\De : X \to X$ as the \emph{modular operator} of our unbounded modular cycle.

An \emph{even unbounded modular cycle} from $\sA$ to $B$ is an odd unbounded modular cycle equipped with a $\zz/2\zz$-grading operator $\ga : X \to X$ such that 
\[
\ga \pi(a) = \pi(a) \ga \q \ga \De = \De \ga \q \mbox{and} \q \ga D = - D \ga
\]
for all $a \in A$.
%For $T \in \pi(\sA) +\cc \cd \T{Id}_X$ we denote the bounded extension of $D T \De - \De T D$ by $d_\De(T) : X \to X$ and remark that $d_\De(T)$ is adjointable with $(d_\De T)^* = -d_\De(T^*)$.
\end{dfn}

\begin{remark}\label{r:twispe}
If we disregard the operator space norm on $\sA$, then the definition of an unbounded Kasparov module (see \cite{BaJu:TBK}) is a special case of the above definition. Indeed, it corresponds exactly to the situation where the modular operator $\De$ is the identity operator on $X$. In fact, given an unbounded Kasparov module $(X,D)$ from $\sA$ to $B$, one may always equip $\sA$ with an operator space norm such that $(X,D,1)$ becomes an unbounded modular cycle from $\sA$ to $B$. %to the case where the modular operator $\De = \T{Id}_X$.

The concept of a twisted spectral triple (see \cite{CoMo:TST}) is closely related to the above definition. Indeed, one of the main examples of a twisted spectral triple is obtained by starting from a unital spectral triple $(\sA,H,D)$ together with a fixed positive and invertible element $g \in \sA$. One then forms the twisted spectral triple $(\sA, H, \ov{ g D g } )$, where the modular automorphism $\si : \sA \to \sA$ is given by $\si(a) := g^2 a g^{-2}$. The triple $(H, \ov{g D g}, g^2)$ is an example of an unbounded modular cycle from $\sA$ to $\cc$ (after choosing an appropriate operator space norm on $\sA$). As explained in \cite{CoMo:TST}, the above procedure provides a noncommutative analogue of the classical operation where an underlying Riemannian metric is changed by a conformal factor, see for example \cite[Proposition 4.3.1]{Hi:CLB}.

Our definition of an unbounded modular cycle is inspired by the notion of a twisted spectral triple, but there are three important differences:
\begin{enumerate}
\item we are considering a bivariant theory, thus the scalars can consist of an arbitrary $C^*$-algebra and not just the complex numbers;
\item the modular operator $\De : X \to X$ can have zero in the spectrum, thus allowing for a treatment of conformal changes of metrics on non-compact manifolds;
\item the modular automorphism $\si$ given by conjugation with $\De$ need not map the algebra $\sA$ into itself, in fact it need not even be defined on $\sA$.
\end{enumerate}

For more information about twisted spectral triples we refer to \cite{KhFa:TST,Mo:LIF,PoWa:NGC}.

There is a link between Hilsum's notion of a half-closed chain and the above notion of an unbounded modular cycle, \cite{Hil:BIK}. Indeed, any half-closed chain $(X,D)$ gives rise to a wealth of unbounded modular cycles via a localization procedure, which uses the methods developed in the present paper, see \cite[Theorem 13]{KaSu:KHC}. 
\end{remark}

Let us spend a little extra time commenting on the conditions in Definition \ref{d:unbkas}. It follows by a density argument that condition $(2)$ and $(3)$ also hold for all $a \in A_1$. For condition $(4)$, we notice that the sequence $\{\De (\De + 1/n)^{-1}\}$ converges strictly to the identity on $X$ (this holds since $\De$ is positive and $\T{Im}(\De)$ is dense in $X$). In general we have that condition $(4)$ is automatic when $\De : X \to X$ is invertible as a bounded adjointable operator. Remark that condition $(2)$ for $a = 0$ and $\la = 1$ says that $\De( \sD(D)) \su \sD(D)$ and that the straight commutator
\[
D \De - \De D : \sD(D) \to X
\]
has a bounded adjointable extension $d(\De) : X \to X$. Notice also the important identity 
\[
d_\De(1) = d(\De)
\]
between straight and twisted commutators.

%For each $\ep \in (0,1]$ we may define a linear map
%\begin{equation}\label{eq:rhonota}
%\rho_{\De,\ep} : \sA \to \sL(X), \q \rho_{\De,\ep}(a) := (\De + \ep)^{-1/2} d_\De(a) (\De + \ep)^{-1/2} .
%\end{equation}
%The second part of condition $(3)$ then says that $\rho_{\De,\ep}$ extends to a completely bounded map $\rho_{\De,\ep} : A_1 \to \sL(X)$ and that the supremum of completely bounded norms
%\[
%\sup_{\ep \in (0,1]}\| \rho_{\De,\ep} \|_{\T{cb}}
%\]

For later use we introduce the following terminology:

\begin{dfn}\label{d:unbkasII} 
When $D : \sD(D) \to X$ is an unbounded selfadjoint and regular operator and $\De : X \to X$ is a bounded positive and selfadjoint operator with dense image we say that a bounded adjointable operator $T : X \to X$ is \emph{differentiable} (with respect to $(D,\De)$) when the following holds: 
\begin{enumerate}
\item $T \De(\sD(D)) \su \sD(D)$ and
\[
D T \De - \De T D : \sD(D) \to X
\]
extends to a bounded adjointable operator $d_\De(T) : X \to X$;
\item the supremum %image of $d_\De(T) : X \to X$ is contained in the image of $\De^{1/2} : X \to X$ and the unbounded operator
\[
\sup_{\ep \in (0,1]}\| (\De + \ep)^{-1/2} d_\De(T) (\De + \ep)^{-1/2} \|_\infty
\]
is finite.
\end{enumerate}
\end{dfn}

We remark that the adjoint of a differentiable operator $T : X \to X$ is automatically differentiable as well and that $d_\De(T)^* = - d_\De(T^*)$. We introduce the notation
\[
\rho_{\De,\ep}(T) := (\De + \ep)^{-1/2} d_\De(T) (\De + \ep)^{-1/2}
\]
for all $\ep \in (0,1]$.

For an unbounded modular cycle $(X,D,\De)$ from $\sA$ to $B$, we see that $\pi(a) + \la : X \to X$ is differentiable with respect to $(D,\De)$ for all $a \in \sA$ and $\la \in \cc$.

%and $\rho_\De(T)^* = - \rho_\De(T^*)$ are valid.

\subsection{Stabilization of unbounded modular cycles}\label{ss:stamodcyc}
Let us fix an unbounded modular cycle $(X,D,\De)$ from the $*$-algebra $\sA$ to the $C^*$-algebra $B$. We let $\ga : X \to X$ denote the grading operator in the even case.
\medskip

The aim of this subsection is to construct a \emph{stabilization} of $(X,D,\De)$ which is an unbounded modular cycle from the finite matrices over $\sA$ to $B$. The parity of the stabilization is the same as the parity of $(X,D,\De)$.
\medskip

To this end, we first notice that the finite matrices over $\sA$ comes equipped with a canonical operator space norm and a canonical $C^*$-norm (see Definition \ref{d:staope}):
\[
\| \cd \|_1 \, , \, \, \| \cd \| : M(\sA) \to [0,\infty) .
\]
The respective completions are the operator space $K(A_1)$ and the $C^*$-algebra $K(A)$. We remark that $K(A)$ is isomorphic to the compact operators on the standard module $\ell^2(\nn,A)$, where $A$ is considered as a Hilbert $C^*$-module over itself.

We now consider the standard module $\ell^2(\nn,X)$ over $B$ and we equip it with the $*$-homomorphism $K(\pi) : K(A) \to \sL(\ell^2(\nn,X))$ given by
\begin{equation}\label{eq:stabkay}
K(\pi)\Big( \sum_{n,m = 1}^\infty a_{nm} \cd \de_{nm} \Big)( \sum_{k = 1}^\infty x_k \de_k)
:= \sum_{n = 1}^\infty \Big( \sum_{m = 1}^\infty \pi(a_{nm})(x_m) \Big) \de_n ,
\end{equation}
where $a \cd \de_{nm} \in K(A)$ denotes the finite matrix with $a \in A$ in position $(n,m)$ and zeroes elsewhere. 

Furthermore, on the standard module over $X$, we have the diagonal operators induced by the unbounded selfadjoint and regular operator $D : \sD(D) \to X$ and the modular operator $\De : X \to X$. The diagonal operator induced by $D : \sD(D) \to X$ is given by
\[
\T{diag}(D) : \sD\big( \T{diag}(D) \big) \to \ell^2(\nn,X) , \q \sum_{n = 1}^\infty x_n \de_n \mapsto \sum_{n = 1}^\infty D(x_n) \de_n ,
\]
where the domain $\sD\big( \T{diag}(D) \big) \su \ell^2(\nn,X)$ is defined by
\[
\sD\big( \T{diag}(D) \big) := \Big\{ \sum_{n = 1}^\infty x_n \de_n \in \ell^2(\nn,X)
\mid x_n \in \sD(D) \, \, \T{and} \, \, \sum_{n = 1}^\infty D(x_n) \de_n \in \ell^2(\nn,X) \Big\} .
\]
The diagonal operator induced by $\De : X \to X$ is given by
\[
\T{diag}(\De) : \ell^2(\nn,X) \to \ell^2(\nn,X) , \q \sum_{n = 1}^\infty x_n \de_n \mapsto \sum_{n = 1}^\infty \De(x_n) \de_n .
\]
Likewise (in the even case), we have the diagonal operator $\T{diag}(\ga) : \ell^2(\nn,X) \to \ell^2(\nn,X)$ induced by the grading operator $\ga : X \to X$.

The unbounded operator $\T{diag}(D) : \sD\big( \T{diag}(D) \big) \to \ell^2(\nn,X)$ is again a selfadjoint and regular operator, indeed the relevant resolvents are the diagonal operators $\T{diag}\big( (D \pm i)^{-1}\big) : \ell^2(\nn,X) \to \ell^2(\nn,X)$. We also note that $\T{diag}(D)$ has a core given by the algebraic direct sum $\op_{n = 1}^\infty \sD(D) \su \ell^2(\nn,X)$. Clearly, $\T{diag}(\De) : \ell^2(\nn,X) \to \ell^2(\nn,X)$ is again positive and selfadjoint with dense image.

To ease the notation, we put
\[
1 \ot D := \T{diag}(D), \q 1 \ot \De := \T{diag}(\De), \q \T{and} \q  1 \ot \ga := \T{diag}(\ga) .
\]

\begin{dfn}
By the \emph{stabilization} of $(X,D,\De)$ we understand the triple $(\ell^2(\nn,X), 1 \ot D, 1 \ot \De)$ with $\zz/2\zz$-grading operator $1 \ot \ga$ in the even case.
\end{dfn}

\begin{prop}\label{p:stamodcyc}
The stabilization $(\ell^2(\nn,X), 1 \ot D, 1 \ot \De)$ is an unbounded modular cycle from $M(\sA)$ to $B$ of the same parity as $(X,D,\De)$.
\end{prop}
\begin{proof}
Since the statement about the grading is clear in the even case, we only need to verify the conditions $(1)$-$(4)$ in Definition \ref{d:unbkas}. We suppress the $*$-homomorphism $K( \sL(X) ) \to \sL( \ell^2(\nn,X))$ induced by the identity $\sL(X) \to \sL(X)$ throughout this proof, see Equation \eqref{eq:stabkay}.

$(1)$ and $(4)$: this follows by standard arguments.

$(2)$: for elements $\sum_{n,m = 1}^N a_{nm} \de_{nm} \in M(\sA)$ and $\la \in \cc$ we record that
\[
d_{1 \ot \De}\big( \sum_{n,m = 1}^N a_{nm} \de_{nm}  + \la  \big)
= \sum_{n,m = 1}^N d_\De(a_{nm}) \de_{nm} + 1 \ot d_\De(\la) .
\]
%where we are suppressing the  This $*$-
%
%The general result then follows by the density of $M(\sA)$ in $\sK(\sA)$ and the fact that $(d_\De \ci \pi) : \sA \to \sL(X)$ and the inclusion $\sA \to A$ are completely bounded.

$(3)$: the first assertion in $(3)$ follows since
\[
\begin{split}
& \| (1 \ot \De + \ep)^{-1/2} d_{1 \ot \De}(\sum_{n,m = 1}^N a_{nm} \de_{nm} + \la) (1 \ot \De + \ep)^{-1/2}  \|_\infty \\
& \q = \| \sum_{n,m = 1}^N (\De + \ep)^{-1/2} d_\De(a_{nm}) (\De + \ep)^{-1/2} \de_{nm} \\ 
& \qqq + 1 \ot (\De + \ep)^{-1/2} d_\De(\la) (\De + \ep)^{-1/2} \|_\infty \\
& \q \leq \sum_{n,m = 1}^N \| (\De + \ep)^{-1/2} d_\De(a_{nm}) (\De + \ep)^{-1/2}\|_\infty \\ 
& \qqq+  \| (\De + \ep)^{-1/2} d_\De(\la) (\De + \ep)^{-1/2} \|_\infty 
\end{split}
\]
for all $\ep \in (0,1]$, $\sum_{n,m = 1}^N a_{nm} \de_{nm} \in M(\sA)$ and $\la \in \cc$.

For the second assertion in $(3)$, we let $\sum_{n,m = 1}^N a_{nm} \de_{nm} \in M(\sA)$ and $\ep \in (0,1]$ be given and notice that
%is clear for $T = \sum_{n,m = 1}^N \pi(a_{nm}) \de_{nm} + \la \cd \T{Id}_{\ell^2(\nn,X)}$. Furthermore, for $\sum_{n,m = 1}^N \pi(a_{nm}) \de_{nm} \in M(\sA)$ we see that
\[
\rho_{1 \ot \De,\ep}( \sum_{n,m = 1}^N a_{nm} \de_{nm}) =  \sum_{n,m = 1}^N \rho_{\De,\ep}(a_{nm}) \de_{nm}
=  K(\rho_{\De,\ep}) \big( \sum_{n,m = 1}^N a_{nm} \de_{nm} \big),
\]
where $K(\rho_{\De,\ep}) : M(\sA) \to K( \sL( X))$ is the completely bounded map induced by $\rho_{\De,\ep} : \sA \to \sL(X)$, see Subsection \ref{ss:stabspa}. 
%
% and where we are suppressing the $*$-homomorphism   %The general result then follows by the density of $M(\sA)$ in $\sK(\sA)$ and the fact that $\rho_\De : \sA \to \sL(X)$ is completely bounded.
%
%$(4)$: This is clear since $\rho_{1 \ot \De} : M(\sA) \to \sL(\ell^2(\nn,X))$ agrees with the induced map $M(\rho_\De) : M(\sA) \to \sL(\ell^2(\nn,X))$ (see the proof of $(3)$).
%
%$(5)$: Let $\{V_n\}$ be a countable approximate unit for $C^*(\De)$ such that $\pi(a) V_n \to \pi(a)$ in operator norm for all $a \in A$. The sequence $\{1 \ot V_n\}_{n = 1}^\infty$ is then a countable approximate unit for $C^*(1 \ot \De)$ such that $K(\pi)(T) (1 \ot V_n) \to K(\pi)(T)$ in operator norm for all $T \in K(A)$.
\end{proof}

%\begin{remark}
%We would like to emphasize that, if we were to consider $\sA$ only as a pre-Banach space and not as a pre-operator space, there would be \emph{no canonical} stabilization available of the unbounded modular cycle $(X,D,\De)$. This is one of the reasons for considering operator spaces instead of Banach spaces.
%\end{remark}

\section{Differentiable Hilbert $C^*$-modules}\label{s:difhil}
Throughout this section $\sA$ and $\sB$ will be $*$-algebras which satisfy the conditions in Assumption \ref{a:alg}. We let $A_1$ and $B_1$ denote the operator space completions and we let $A$ and $B$ denote the $C^*$-completions of $\sA$ and $\sB$, respectively.
\medskip

The next definition is the second main new concept which we introduce in this paper:

\begin{dfn}\label{d:difhil}
A Hilbert $C^*$-module $X$ over $B$, which comes equipped with a $*$-homomorphism $\pi : A \to \sL(X)$ is said to be \emph{differentiable} (from $\sA$ to $\sB$) when there exists a sequence $\{ \xi_n \}_{n = 1}^\infty$ in $X$ such that the following holds:
\begin{enumerate}
\item $\T{span}_{\cc} \{ \xi_n \cd b \mid b \in B \, , \, \, n \in \nn \}$ is norm-dense in $X$;
\item $\inn{\xi_n, (\pi(a) + \la) \xi_m} \in \sB$ for all $a \in \sA$, $\la \in \cc$ and $n,m \in \nn$;
\item the sequence of finite matrices
\[
\Big\{ \sum_{n,m = 1}^N \binn{\xi_n, (\pi(a) + \la) \xi_m} \de_{nm} \Big\}_{N = 1}^\infty
\]
is a Cauchy sequence in $K(B_1)$ for all $a \in \sA$ and $\la \in \cc$;
\item the linear map $\tau : \sA \to K(B_1)$, $a \mapsto \sum_{n,m = 1}^\infty \binn{\xi_n, \pi(a) \xi_m } \de_{nm}$ is completely bounded (with respect to the operator space norm on $\sA$).
\end{enumerate}
We refer to a sequence $\{ \xi_n \}_{n = 1}^\infty$ in $X$ satisfying the above conditions as a \emph{differentiable generating sequence}.
\end{dfn}
%
%As in the case of an unbounded modular cycle we let $\pi : A_1 \to \sL(X)$ denote the completely bounded map induced by the inclusion $\sA \to A$ and $\pi : A \to \sL(X)$. We then obtain (by a density argument) that $(3)$ holds for all $T \in \pi(A_1)$. Likewise, we have that the completely bounded extension $\tau : A_1 \to \sK(B_1)$ is given explicitly by $\tau(a) = \sum_{n,m = 1}^\infty \binn{\xi_n, \pi(a) \xi_m} \de_{nm}$ for all $a \in A_1$.

%\medskip

%From now on we suppose that $X$ is a differentiable Hilbert $C^*$-module and we let $\{ \xi_n \}_{n = 1}^\infty$ be a fixed generating sequence satisfying the conditions in Definition \ref{d:difhil}.

\begin{remark}
The conditions $(3)$ and $(4)$ in Definition \ref{d:difhil} can be replaced by the following:

$(3a)$: the sequence of finite matrices
\[
\Big\{ \sum_{n,m = 1}^N \binn{\xi_n, (\pi(a) + \la) \xi_m} \de_{nm} \Big\}_{N = 1}^\infty
\]
is bounded in $K(B_1)$ for all $a \in \sA$ and $\la \in \cc$;

$(4a)$: the linear map $\sA \to M_\nn(B_1)$, $a \mapsto \sum_{n,m = 1}^\infty \binn{\xi_n, \pi(a) \xi_m } \de_{nm}$ is completely bounded, where $M_\nn(B_1)$ is the operator space of infinite matrices over $B_1$, see \cite[Section 1.2.26]{BlMe:OMO} for details.

Given a sequence $\{\xi_n\}$ that satisfies $(1)$, $(2)$, $(3a)$, and $(4a)$ we obtain a sequence satisfying $(1)$, $(2)$, $(3)$, and $(4)$ by rescaling each $\xi_n \in X$ by $\frac{1}{n}$, for example.
\end{remark}

\subsubsection{Example: Finitely generated Hilbert $C^*$-modules}
Let us consider a $*$-algebra $\sB$ which satisfies the conditions of Assumption \ref{a:alg}. Let us also consider a dense $*$-subalgebra $\sA$ of a $C^*$-algebra $A$. We do not assume that $\sA$ satisfies Assumption \ref{a:alg}. Let now $X$ be a \emph{finitely generated} Hilbert $C^*$-module over $B$ with generators $\xi_1,\ldots,\xi_N \in X$ and let $\pi : A \to \sL(X)$ be a $*$-homomorphism. By ``finitely generated'' we mean that the subspace
\[
\big\{ \xi_n \cd b \mid n \in \{1,\ldots,N\} \, , \, \, b \in B \big\} \su X
\]
is dense in the norm-topology on $X$. We emphasize that this condition does not at all imply that $X$ is finitely generated projective as a right module over $B$: consider for example $C_0( (0,1))$ as a Hilbert $C^*$-module over $C_0(\rr)$.

Suppose now that
\[
\inn{\xi_n, (\pi(a) + \la) \xi_m} \in \sB \q \T{for all } a \in \sA , \, \,  \la \in \cc \, \, \mbox{and} \, \, \, n,m \in \{1,\ldots,N\} .
\]
We then have a linear map
\[
\tau : \sA \to M_N(\sB), \q \tau(a) := \sum_{n,m = 1}^N \inn{\xi_n, \pi(a) \xi_m} \de_{nm} ,
\]
from which we can obtain an operator space norm on $\sA$ by defining
\[
\| a \|_1 := \max\{ \| a \| , \| \tau(a) \|_1 \} \q \T{for all } a \in M_k(\sA) , \, \, k \in \nn ,
\]
where we have suppressed the usual identification $M_k\big( M_N(\sB) \big) \cong M_{k \cd N}(\sB)$ (see Definition \ref{d:staope}). By construction we get that $X$ is a differentiable Hilbert $C^*$-module from $\sA$ to $\sB$.

\section{The modular lift}\label{s:modlif}
In this section we consider two Hilbert $C^*$-modules $X$ and $Y$ with the same base $C^*$-algebra $A$. We fix an unbounded selfadjoint and regular operator $D : \sD(D) \to Y$ on the Hilbert $C^*$-module $Y$ together with a bounded selfadjoint and positive operator $\Ga : Y \to Y$ with dense image. Furthermore, we consider a bounded adjointable operator $\Phi : X \to Y$ such that the adjoint $\Phi^* : Y \to X$ has dense image.
\medskip

\emph{The main concern of this section is to ``transport'' the unbounded selfadjoint and regular operator $D : \sD(D) \to Y$ to an unbounded selfadjoint and regular operator $D_\De : \sD(D_\De) \to X$. This transportation will happen via the bounded adjointable operator $\Phi : X \to Y$.}
\medskip

We apply the notation:
\[
\De := \Phi^* \Ga \Phi : X \to X \q \T{and} \q G := \Phi \Phi^* : Y \to Y .
\]
Remark that $\De : X \to X$ is bounded selfadjoint and positive and that $\T{Im}(\De) \su X$ is norm-dense.
\medskip

\emph{We assume that $G$ is differentiable with respect to $(D,\Ga)$, see Definition \ref{d:unbkasII}}.
\medskip

%The following standing assumptions will be in effect:
%
%\begin{assum}\label{a:diff}
%It is assumed that
%\begin{enumerate}
%\item the bounded adjointable operator $G \Ga : Y \to Y$ has $\sD(D) \su Y$ as an invariant submodule;
%\item the twisted commutator
%The composition $G : \Phi \Phi^* : Y \to Y$ has Thus, $\Phi \Phi^*$ has
%\[
%D G \Ga - \Ga G D : \sD(D) \to Y
%\]
%has a bounded extension to $Y$. This bounded extension is denoted by $d_\Ga(G) : Y \to Y$;
%\item the supremum of operator norms 
%\[
%\sup_{\ep \in (0,1]}\big\| (\Ga + \ep)^{-1/2} d_\Ga(G) (\Ga + \ep)^{-1/2} \big\|_\infty
%\]
%is finite.
%\end{enumerate}
%\end{assum}

For each $\ep \in (0,1]$, we recall the notation:
\[
\rho_{\Ga,\ep}(G) := (\Ga + \ep)^{-1/2} d_\Ga(G) (\Ga + \ep)^{-1/2} : Y \to Y .
\]
%Likewise we have that $\rho_\Ga(G) : Y \to Y$ is adjointable with $\rho_\Ga(G)^* = - \rho_\Ga(G)$.
\medskip

\emph{The main aim of this section is to show that the composition
\[
\Phi^* D \Phi : \sD(\Phi^* D \Phi) \to X
\]
is essentially selfadjoint and regular, where the domain is given by
\[
\sD(\Phi^* D \Phi) := \sD(D \Phi) := \big\{ x \in X \mid \Phi(x) \in \sD(D) \big\} .
\] }
\medskip

We immediately remark that $\sD(\Phi^* D \Phi) \su X$ is norm-dense. Indeed, this follows since $\Phi^* \Ga (\sD(D)) \su \sD(\Phi^* D \Phi)$. Furthermore, it is evident that the unbounded operator $\Phi^* D \Phi : \sD(\Phi^* D \Phi) \to X$ is symmetric. 

We notice that $\De\big( \sD(\Phi^* D \Phi) \big) \su \sD(\Phi^* D \Phi)$ and that
\[
(\Phi^* D \Phi \De - \De \Phi^* D \Phi)(\eta) = (\Phi^* D \Phi \Phi^* \Ga \Phi - \Phi^* \Ga \Phi \Phi^* D \Phi)(\eta) 
= (\Phi^* d_\Ga(G) \Phi)(\eta)
\]
for all $\eta \in \sD(\Phi^* D \Phi)$. In particular, this shows that the \emph{straight} commutator
\[
\Phi^* D \Phi \De - \De \Phi^* D \Phi : \sD(\Phi^* D \Phi) \to X
\]
has a bounded adjointable extension to $X$.

%To ease the notation, we let
%\[
%G := \Phi \Phi^* \in \sL(Y) \q \T{and} \q \De := \Phi^* \Phi \in \sL(X)
%\]

\begin{dfn}
The \emph{modular lift} of $D : \sD(D) \to Y$ with respect to $\Phi : X \to Y$ is the closure of $\Phi^* D \Phi : \sD(\Phi^* D \Phi) \to X$. The modular lift is denoted by $D_\De : \sD(D_\De) \to X$.
\end{dfn}

\subsection{Selfadjointness}
In order to show that the modular lift is selfadjoint we need a few preliminary lemmas.

\begin{lemma}\label{l:modadj}
Let $\xi \in \sD\big( (D_\De)^* \big)$. Then $\De(\xi) \in \sD\big(  \Phi^* D \Phi \big)$ and
\[
(\Phi^* D \Phi)(\De \xi) = \De (D_\De)^*(\xi) + \Phi^* d_\Ga(G) \Phi(\xi) .
\]
\end{lemma}
\begin{proof}
Let $\eta \in \sD(D)$ and compute as follows:
\[
\begin{split}
\inn{\Phi \De(\xi), D(\eta)}
& = \inn{\Phi(\xi), \Ga G D (\eta)} 
= \inn{\Phi(\xi), D G \Ga (\eta)} 
- \inn{\Phi(\xi), d_\Ga(G)(\eta)} \\
& = \inn{(D_\De)^*(\xi), \Phi^* \Ga (\eta)} - \inn{d_\Ga(G)^*\Phi(\xi),\eta} \\
& = \inn{\Ga \Phi(D_\De)^*(\xi),\eta} + \inn{d_\Ga(G)\Phi(\xi),\eta} .
\end{split}
\]
Using the selfadjointness assumption on $D : \sD(D) \to Y$, this implies that $\Phi \De(\xi) \in \sD(D)$ and furthermore that
\[
D (\Phi \De \xi) = \Ga \Phi(D_\De)^*(\xi) + d_\Ga(G)\Phi(\xi) .
\]
This proves the lemma.
\end{proof}

\begin{lemma}\label{l:modadjinv}
Let $\xi \in \sD\big( (D_\De)^* \big)$ and let $z \in \cc \sem [0,\infty)$ be given. Then $(\De - z)^{-1}(\xi) \in \sD\big( (D_\De)^* \big)$ and
\[
\begin{split}
(D_\De)^* (\De - z)^{-1}(\xi) 
& = (\De - z)^{-1} (D_\De)^* (\xi) \\ 
& \qq - (\De - z)^{-1} \Phi^* d_\Ga(G) \Phi (\De - z)^{-1}(\xi) .
\end{split}
\]
\end{lemma}
\begin{proof}
We consider the smallest $C^*$-subalgebra $C^*(\De) \su \sL(X)$ containing $\De \in \sL(X)$ together with the $*$-subalgebra $\sD(\de) \su C^*(\De)$ defined by
\[
\begin{split}
\big( T \in \sD(\de)  \big) \lrar \Big( 
& T \in C^*(\De) \, , \, \,  T\big( \sD((D_\De)^*) \big) \su \sD(D_\De) \, \T{ and } \\
& \, \, \, \,  T (D_\De)^* - D_\De T : \sD((D_\De)^*) \to X \\ 
& \, \, \T{ has a bounded adjointable extension } \de(T) : X \to X \Big).
\end{split}
\]
It follows by Lemma \ref{l:modadj} that $\sD(\de) \su C^*(\De)$ is norm-dense. Moreover, the linear map $\de : \sD(\de) \to \sL(X)$ is a closed densely defined derivation on $C^*(\De)$ and $\sD(\de)$ becomes a Banach $*$-algebra when equipped with the norm $\| \cd \|_\de : \sD(\de) \to [0,\infty)$ defined by $\| T \|_\de : = \| T \|_\infty + \|\de(T) \|_\infty$. As a consequence of \cite[Proposition 3.12]{BlCu:DNS}, the inclusion $\sD(\de) \su C^*(\De)$ is spectrally invariant and it holds in particular that $(\De - z)^{-1}$ is an element in the unital $*$-subalgebra $\sD(\de) + \cc \cd 1 \su \sL(X)$. This implies that $(\De - z)^{-1}(\xi) \in \sD\big( (D_\De)^* \big)$. The explicit formula for the commutator
\[
\big[ (D_\De)^*, (\De - z)^{-1} \big] : \sD\big( (D_\De)^* \big) \to X 
\]
is now an algebraic consequence of Lemma \ref{l:modadj}.
\end{proof}

We are now ready to show that the modular lift $D_\De : \sD(D_\De) \to X$ is selfadjoint:

\begin{prop}\label{p:modseladj}  
Suppose that $\Phi^* : Y \to X$ has dense image and that $G = \Phi \Phi^* : Y \to Y$ is differentiable with respect to $(D,\Ga)$. Then the composition
\[
\Phi^* D \Phi : \sD(\Phi^* D \Phi) \to X
\]
is essentially selfadjoint.
\end{prop}
\begin{proof}
It is enough to prove that $\sD\big( (D_\De)^*\big) \su \sD(D_\De)$. Thus, let $\xi \in \sD\big( (D_\De)^* \big)$ be given. 

Let us consider the sequence $\big\{ \De (\De + 1/n)^{-1}(\xi) \big\}$ and recall that
\[
\De (\De + 1/n)^{-1}(\xi) \to \xi .
\]
Furthermore, by Lemma \ref{l:modadj} and Lemma \ref{l:modadjinv} we have that $\De (\De + 1/n)^{-1}(\xi) \in \sD\big( \Phi^* D \Phi \big)$ for all $n \in \nn$.

To show that $\xi \in \sD(D_\De)$ it therefore suffices to prove that the sequence
\[
\big\{ (\Phi^* D \Phi)\De(\De + 1/n)^{-1}(\xi) \big\}
\]
is norm-convergent in $X$.

For each $n \in \nn$ we use Lemma \ref{l:modadj} and Lemma \ref{l:modadjinv} to compute in the following way:
\[
\begin{split}
& (\Phi^* D \Phi)\De(\De + 1/n)^{-1}(\xi)  \\
& \q = \De (D_\De)^* (\De + 1/n)^{-1} (\xi) 
+ \Phi^* d_\Ga(G) \Phi (\De + 1/n)^{-1} (\xi) \\
& \q = \De (\De + 1/n)^{-1} (D_\De)^*(\xi) 
- \De (\De + 1/n)^{-1} \Phi^* d_\Ga(G) \Phi (\De + 1/n)^{-1}(\xi)  \\
& \qqq + \Phi^* d_\Ga(G) \Phi (\De + 1/n)^{-1} (\xi) \\
& \q = \De (\De + 1/n)^{-1} (D_\De)^*(\xi) 
+ \frac{1}{n} (\De + 1/n)^{-1} \Phi^* d_\Ga(G) \Phi (\De + 1/n)^{-1} (\xi) .
\end{split}
\]

Since the sequence $\big\{ \De (\De + 1/n)^{-1} \big\}$ converges strictly to the identity operator on $X$, the result of the proposition is proved, provided that the sequence
\[
\big\{ \frac{1}{n} (\De + 1/n)^{-1} \Phi^* d_\Ga(G) \Phi (\De + 1/n)^{-1} (\xi) \big\}
\]
converges to zero in the norm on $X$. But this is a consequence of the next lemma.
\end{proof}

\begin{lemma}\label{l:strlimzer} 
The sequence
\[
\big\{ \frac{1}{n} (\De + 1/n)^{-1} \Phi^* d_\Ga(G) \Phi (\De + 1/n)^{-1} \big\}
\]
is bounded in operator norm and converges strictly to the zero operator on $X$.
\end{lemma}
\begin{proof}
We first show that our sequence is bounded in operator norm. To this end, let $\xi \in X$ and $n \in \nn$ and notice that
\[
\begin{split}
& \frac{1}{n} \big\|  (\De + 1/n)^{-1} \Phi^* d_\Ga(G) \Phi (\De + 1/n)^{-1} (\xi) \big\| \\
& \q = \frac{1}{n} \lim_{\ep \searrow 0} \big\|  (\De + 1/n)^{-1} \Phi^* \Ga^{1/2} \rho_{\Ga,\ep}(G) \Ga^{1/2} \Phi (\De + 1/n)^{-1} (\xi) \big\| \\
& \q \leq 
\frac{1}{n} \cd \big\|  \Ga^{1/2} \Phi (\De + 1/n)^{-1} \|_\infty^2 \cd \sup_{\ep \in (0,1]}\| \rho_{\Ga,\ep}(G) \|_\infty \cd \| \xi \|
\leq \sup_{\ep \in (0,1]}\| \rho_{\Ga,\ep}(G) \|_\infty \cd \| \xi \| .
\end{split}
\]

To prove the lemma, we may now limit ourselves to showing that
\[
\frac{1}{n} (\De + 1/n)^{-1} \Phi^* d_\Ga(G) \Phi (\De + 1/n)^{-1}(\xi) \to 0
\]
for all $\xi$ in a dense subspace of $X$. Since $\T{Im}(\De) \su X$ is dense in $X$ we let $\eta \in X$ and remark that
\[
\begin{split}
& \big\| \frac{1}{n} (\De + 1/n)^{-1} \Phi^* d_\Ga(G) \Phi (\De + 1/n)^{-1}\De(\eta) \big\| \\
& \q \leq \frac{1}{n} \cd \big\| (\De + 1/n)^{-1} \Phi^* \Ga^{1/2} \big\|_\infty \cd  \big\| \Ga^{1/2} \Phi \De (\De + 1/n)^{-1}(\eta) \big\| 
\cd \sup_{\ep \in (0,1]}\| \rho_{\Ga,\ep}(G) \|_\infty \\
& \q \leq \frac{1}{\sqrt{n}} \cd \| \Ga^{1/2} \Phi \|_\infty \cd \| \eta \| \cd \sup_{\ep \in (0,1]}\| \rho_{\Ga,\ep}(G) \|_\infty
\end{split}
\]
for all $n \in \nn$. This computation ends the proof of the present lemma.
\end{proof}

\subsection{Regularity}\label{ss:regula}
In order to show that the modular lift $D_\De : \sD(D_\De) \to X$ is regular we will use the local-global principle for unbounded regular operators, see \cite{Pi:ORC,KaLe:LGR}. We will thus pause for a moment and remind the reader how this principle works.

Let $\rho : A \to \cc$ be a state on the $C^*$-algebra $A$ and define the pairing
\[
\inn{\cd,\cd}_\rho : X \ti X \to \cc \q \inn{x_0,x_1}_\rho := \rho(\inn{x_0,x_1}) .
\]
Putting $N_\rho := \big\{ x \in X \mid \inn{x,x}_\rho = 0 \big\}$, we obtain that the vector space quotient $X/N_\rho$ has a well-defined norm, $\| [x] \|_\rho := \inn{x,x}_\rho^{1/2}$, and the completion of $X/N_\rho$ is a Hilbert space with inner product induced by $\inn{\cd,\cd}_\rho$. We denote this Hilbert space by $X_\rho$ and let $[\cd ] : X \to X_\rho$ denote the canonical map (quotient followed by inclusion). One may also view the Hilbert space $X_\rho$ as an interior tensor product $X_\rho \cong X \hot_A H_\rho$, where the Hilbert space $H_\rho$ is the carrier of the GNS-representation of $A$ associated to the state $\rho : A \to \cc$.

The unbounded selfadjoint operator $D_\De : \sD(D_\De) \to X$ yields an induced unbounded symmetric operator
\[
(D_\De)_\rho : \sD\big( (D_\De)_\rho\big) \to X_\rho \q [x] \mapsto [D_\De(x)] ,
\]
where the domain is given by
\[
\sD\big( (D_\De)_\rho\big) := \big\{ [x] \mid x \in \sD(D_\De) \big\} .
\]
We denote the closure of this unbounded symmetric operator by
\[
D_\De \ot 1 : \sD( D_\De \ot 1) \to X_\rho .
\]

The local-global principle states that the unbounded selfadjoint operator $D_\De$ is regular if and only if $D_\De \ot 1$ is selfadjoint for every state $\rho : A \to \cc$, see \cite[Theorem 4.2]{KaLe:LGR}. We remark that an even stronger result is proved in \cite{Pi:ORC}: it does in fact suffice to prove selfadjointness for every pure state on $A$.
\medskip

Let us from now on fix a state $\rho : A \to \cc$. We are interested in showing that $D_\De \ot 1 : \sD(D_\De \ot 1) \to X_\rho$ is selfadjoint. We remark that it already follows by the local-global principle that the unbounded operator $D \ot 1 : \sD(D \ot 1) \to Y_\rho$ is selfadjoint.

The proof of the next lemma is left as an exercise to the reader (we are defining $\Ga \ot 1 : Y_\rho \to Y_\rho$ and $\Phi \ot 1 : X_\rho \to Y_\rho$ using the same recipe as in the unbounded case):

\begin{lemma}\label{l:loctri}
The bounded operator $\Phi^* \otimes 1 : Y_\rho \to X_{\rho}$ has dense image and $G \otimes 1 = \Phi \Phi^* \otimes 1 : Y_\rho \to Y_\rho$ is differentiable with respect to $(D \ot 1, \Ga \ot 1)$. Furthermore, we have the identities
\[
d_{\Ga \ot 1}(G \ot 1) = d_\Ga(G) \ot 1 \q \mbox{and} \q \rho_{\Ga \ot 1,\ep}(G \ot 1) = \rho_{\Ga,\ep}(G) \ot 1
\]
for all $\ep \in (0,1]$.
\end{lemma}

It is a consequence of the above lemma and Proposition \ref{p:modseladj} that the composition
\[
(\Phi^* \ot 1)(D \ot 1)(\Phi \ot 1) : \sD\big( (D \ot 1)(\Phi \ot 1) \big) \to X_\rho
\]
is essentially selfadjoint. We denote the closure by $(D \ot 1)_{\De \ot 1}$ and focus our attention on proving the identity
\[
(D \ot 1)_{\De \ot 1} = D_\De \ot 1 ,
\]
which will then imply the regularity of the modular lift $D_{\De}$.

We start by proving the easiest of the two inclusions (required for establishing the above identity):

\begin{lemma}\label{l:incloc}
\[
D_\De \ot 1 \su (D \ot 1)_{\De \ot 1} .
\]
\end{lemma}
\begin{proof}
Let $\xi \in \sD(D_\De \ot 1)$. Then there exists a sequence $\{ \xi_n \}$ in $\sD(\Phi^* D \Phi)$ such that
\[
[\xi_n] \to \xi \q \T{and} \q [D_\De(\xi_n)] \to (D_\De \ot 1)(\xi) .
\]
But then we clearly have that $[\xi_n] \in \sD\big( (\Phi^* \ot 1)(D \ot 1)(\Phi \ot 1)\big)$ and furthermore that
\[
(\Phi^* \ot 1)(D \ot 1)(\Phi \ot 1)[\xi_n] = [D_\De(\xi_n)] .
\]
This proves the lemma.
\end{proof}

The proof of the reverse inclusion
\begin{equation}\label{eq:modrevinc}
(D \ot 1)_{\De \ot 1} \su D_\De \ot 1
\end{equation}
is more subtle. It will rely on the following lemma:

\begin{lemma}\label{l:phista}
Let $\xi \in \sD\big((D \ot 1)_{\De \ot 1}\big)$. Then $(\De \ot 1)(\xi) \in \sD(D_\De \ot 1)$ and furthermore,
\[
(D_\De \ot 1)(\De \ot 1)(\xi) = (\De  \ot 1)(D \ot 1)_{\De \ot 1}(\xi) + (\Phi^* d_\Ga(G) \Phi \ot 1)(\xi) .
\]
\end{lemma}
\begin{proof}
Let first $\eta \in \sD(D \ot 1)$ be given. Choose a sequence $\{\eta_n\}$ in $\sD(D)$ such that
\[
[\eta_n] \to \eta \q \T{and} \q [D\eta_n] \to (D \ot 1)(\eta) .
\]
Remark that $(\Phi^* \Ga \ot 1)[\eta_n] \in \sD\big( (D_\De)_\rho \big)$ for all $n \in \nn$ and furthermore that
\[
(\Phi^* \Ga \ot 1)[\eta_n] \to (\Phi^* \Ga \ot 1)(\eta) .
\]
We now compute as follows:
\[
(D_\De)_\rho[\Phi^* \Ga \eta_n] = [\Phi^* D G \Ga \eta_n] = [\Phi^* \Ga G D \eta_n] + [\Phi^* d_\Ga(G) \eta_n] .
\]
This shows that
\[
(D_\De)_\rho(\Phi^* \Ga \ot 1)[\eta_n] \to (\Phi^* \Ga G  \ot 1)(D \ot 1)(\eta) + (\Phi^* d_\Ga(G) \ot 1)(\eta) .
\]
We thus have that $(\Phi^* \Ga \ot 1)(\eta) \in \sD(D_\De \ot 1)$ and furthermore that
\[
(D_\De \ot 1)(\Phi^* \Ga \ot 1)(\eta) = (\Phi^* \Ga G  \ot 1)(D \ot 1)(\eta) + (\Phi^* d_\Ga(G) \ot 1)(\eta) .
\]

Let now $\xi \in \sD\big( (D \ot 1)(\Phi \ot 1) \big)$. It then follows from the above that $(\De \ot 1)(\xi) \in \sD(D_\De \ot 1)$ and moreover that
\[
\begin{split}
(D_\De \ot 1)(\De \ot 1)(\xi) & = (D_\De \ot 1)(\Phi^* \Ga \ot 1)(\Phi \ot 1)(\xi)  \\
& =  (\Phi^* \Ga G  \ot 1)(D \ot 1)(\Phi \ot 1)(\xi) + (\Phi^* d_\Ga(G) \Phi \ot 1)(\xi) \\
& = (\De \ot 1)(D \ot 1)_{\De \ot 1}(\xi) + (\Phi^* d_\Ga(G) \Phi \ot 1)(\xi) .
\end{split}
\]

The result of the lemma now follows by using that $\ov{ (\Phi^* \ot 1)(D \ot 1)(\Phi \ot 1) } = (D \ot 1)_{\De \ot 1}$ by definition.
\end{proof}

We are now ready to prove the reverse inclusion which (together with Lemma \ref{l:incloc}) will imply the following:

\begin{prop}\label{p:locmodsel}
We have the identity of unbounded operators
\[
(D \ot 1)_{\De \ot 1} =  D_\De \ot 1
\]
on the Hilbert space $X_\rho$. In particular, we obtain that $D_\De \ot 1$ is selfadjoint.
\end{prop}
\begin{proof}
By Lemma \ref{l:incloc} we only need to show that
\[
\sD\big( (D \ot 1)_{\De \ot 1} \big) \su \sD(D_\De \ot 1)  .
\]

Let thus $\xi \in \sD\big( (D \ot 1)_{\De \ot 1} \big)$ be given. For each $n \in \nn$, it is a consequence of Lemma \ref{l:modadjinv} and Lemma \ref{l:phista} that
\[
\big( \De (\De + 1/n)^{-1} \ot 1 \big)(\xi) \in \sD(D_\De \ot 1) .
\]
Furthermore, these two lemmas allow us to compute as follows:
\[
\begin{split}
& (D_\De \ot 1)\big( \De (\De + 1/n)^{-1} \ot 1 \big)(\xi) \\
& \q = (\De \ot 1) (D \ot 1)_{\De \ot 1} \big( (\De + 1/n)^{-1} \ot 1 \big)(\xi) \\
& \qqq  + \big(\Phi^* d_\Ga(G) \Phi (\De + 1/n)^{-1} \ot 1 \big)(\xi) \\
& \q = \big(\De (\De + 1/n)^{-1} \ot 1 \big)(D \ot 1)_{\De \ot 1} (\xi) \\
& \qqq - (\De (\De + 1/n)^{-1} \Phi^* d_\Ga(G) \Phi (\De + 1/n)^{-1} \ot 1 \big)(\xi) \\
& \qqq + \big(\Phi^* d_\Ga(G) \Phi (\De + 1/n)^{-1} \ot 1 \big)(\xi) \\
& \q = 
\big(\De (\De + 1/n)^{-1} \ot 1 \big)(D \ot 1)_{\De \ot 1} (\xi)  \\ 
& \qqq + \big( 1/n (\De + 1/n)^{-1} \Phi^* d_\Ga(G) \Phi (\De + 1/n)^{-1} \ot 1 \big)(\xi) .
\end{split}
\]

Together with Lemma \ref{l:strlimzer} this computation shows that
\[
(D_\De \ot 1) \big( (\De + 1/n)^{-1} \De \ot 1\big)(\xi)  \to (D \ot 1)_{\De \ot 1} (\xi) .
\]
This proves the present proposition. Indeed, the remaining fact that $D_\De \ot 1 : \sD(D_\De \ot 1) \to X_\rho$ is selfadjoint follows immediately since $(D \ot 1)_{\De \ot 1}$ is selfadjoint (see Lemma \ref{l:loctri} and Proposition \ref{p:modseladj}).
\end{proof}

The main theorem of this section is now a consequence of the local-global principle and Proposition \ref{p:locmodsel} (recall that differentiability is introduced in Definition \ref{d:unbkasII}):

\begin{thm}\label{t:modselreg}
Suppose that $\Phi^* : Y \to X$ has dense image and that $\Phi \Phi^* : Y \to Y$ is differentiable with respect to $(D,\Ga)$. Then the modular lift $D_\De : \sD(D_\De) \to X$ is selfadjoint and regular.
\end{thm}

\section{Compactness of resolvents}
We will in this section remain in the general setting presented in Section \ref{s:modlif}. We will thus assume that $\Phi^* : Y \to X$ has dense image and that $G = \Phi \Phi^* : Y \to Y$ is differentiable with respect to $(D,\Ga)$. In particular, it follows by Theorem \ref{t:modselreg} that the modular lift $D_\De := \ov{\Phi^* D \Phi} : \sD(D_\De) \to X$ is an unbounded selfadjoint and regular operator. We recall that $\De := \Phi^* \Ga \Phi : X \to X$.
\medskip

\emph{We are going to study the compactness properties of the resolvent $(i + D_\De)^{-1} : X \to X$ of the modular lift.}
\medskip

\begin{lemma}\label{l:cruxide}
We have the identity
\[
\De^2 (i + D_\De)^{-1} = \Phi^* \Ga (i + D)^{-1} \cd \Big( 
\big( i (G - 1)\Ga  + d_\Ga(G) \big) \Phi (i + D_\De)^{-1} + \Ga \Phi \Big) .
\]
\end{lemma}
\begin{proof}
Let $\xi \in \sD(D \Phi)$. Since the unbounded operator $(i + \Phi^* D \Phi) : \sD(D \Phi) \to X$ has dense image (by Theorem \ref{t:modselreg}) it is enough to verify that
\[
\De^2(\xi)
= 
\Phi^* \Ga (i + D)^{-1} \Big( \big( i (G - 1)\Ga  + d_\Ga(G) \big) \Phi + \Ga \Phi (i + \Phi^* D \Phi)\Big)(\xi) .
\]
But this follows from the computation
\[
\begin{split}
& \Phi^* \Ga (i + D)^{-1} \Big( 
\big( i (G - 1)\Ga  + d_\Ga(G) \big) \Phi + \Ga \Phi (i + \Phi^* D \Phi)\Big)(\xi) \\
& \q = 
\Phi^* \Ga (i + D)^{-1} \big( i G \Ga + D G \Ga \big) \Phi (\xi) \\
& \q =  
\Phi^* \Ga G \Ga \Phi (\xi)
= 
\De^2 (\xi) . \qedhere
\end{split}
\]
\end{proof}

We let $C^*(\De) \su \sL(X)$ denote the smallest $C^*$-subalgebra containing $\De : X \to X$.

\begin{prop}\label{p:compacres}
Suppose that $\Phi^* \Ga (i + D)^{-1} \in \sK(Y,X)$. Then $T \cd (i + D_\De)^{-1} \in \sK(X)$ for all $T \in C^*(\De)$.
\end{prop}
\begin{proof}
It is an immediate consequence of Lemma \ref{l:cruxide} that
\[
\De^2 (i + D_\De)^{-1} \in \sK(X) .
\]
The result of the lemma therefore follows by noting that the sequence $\big\{ \De^2 (\De + 1/n)^{-1} \big\}$ converges to $\De : X \to X$ in operator norm.
\end{proof}

For later use, we are also interested in the relationship between the resolvents of the squares $(D_\De)^2 : \sD((D_\De)^2) \to X$ and $D^2 : \sD(D^2) \to Y$. In order to study these two resolvents we need the following extra assumption:

\begin{assum}
It is assumed that $\Ga\big(  \sD(D) \big) \su \sD(D)$ and that the straight commutator
\[
D \Ga - \Ga D : \sD(D) \to Y
\]
has a bounded adjointable extension $d(\Ga) : Y \to Y$. %In other words, we assume that the identity operator $1 \in \sL(X)$ is differentiable with respect to $(D,\Ga)$.
\end{assum}

In the next lemma, we apply the notation $d_\Ga(\Ga G) : Y \to Y$ for the bounded adjointable extension of the twisted commutator
\[
D \Ga G \Ga - \Ga^2 G D = [D, \Ga] G \Ga + \Ga ( D G \Ga - \Ga G D ) : \sD(D) \to Y .
\]
%Indeed, it follows from our assumptions that $\Ga G$ is differentiable with respect to $(D,\Ga)$

Let us fix a constant $r \in (\| \De \|^2_\infty + \| \Ga \|^2_\infty,\infty)$ and apply the notation
\[
T_\la := (1 + \la \Ga^2/r + D^2)^{-1} : Y \to Y \q \T{and} \q
S_\la := (1 + \la \De^2/r + (D_\De)^2)^{-1} : X \to X
\]
for all $\la \geq 0$.

The next result will play an important role in our later investigations of the relationship between the unbounded Kasparov product and the interior Kasparov product:

\begin{prop}\label{p:cruxideII}
We have the identity
\[
\begin{split}
\Phi \De^2 S_\la - T_\la \Ga^2 \Phi
& = T_\la \big( \Phi \De^2 - \Ga^2 \Phi + d_\Ga(\Ga G) \Phi D_\De \big) S_\la \\
& \qqq + (D T_\la)^* \big( d_\Ga(G) G \Ga + \Ga G d_\Ga(G) \big) \Phi S_\la 
\end{split}
\]
for all $\la \geq 0$.
\end{prop}
\begin{proof}
Let $\la \geq 0$ and let $\xi \in \sD((D_\De)^2)$. To prove the proposition, it suffices to check that
\[
\begin{split}
& \Big( \Phi \De^2 - T_\la \Ga^2 \Phi (1 + \la \De^2/r + (D_\De)^2) \Big)(\xi) \\
& \q = T_\la \big( \Phi \De^2 - \Ga^2 \Phi + d_\Ga(\Ga G) \Phi D_\De \big)(\xi) \\
& \qqq + (D T_\la)^* \big( d_\Ga(G) G \Ga + \Ga G d_\Ga(G) \big) \Phi(\xi) .
\end{split}
\]

We claim that
%However, using the techniques presented in the proof of Lemma \ref{l:cruxideI} we see that
\begin{equation}\label{eq:cruxI}
(D T_\la)^* \big( d_\Ga(G) G \Ga + \Ga G d_\Ga(G) \big) \Phi(\xi)
= (D^2 T_\la)^* \Phi \De^2(\xi) - (D T_\la)^* \Ga \Phi \De D_\De(\xi)
\end{equation}
and moreover that
\begin{equation}\label{eq:cruxII}
T_\la d_\Ga(\Ga G) \Phi (\xi) = (D T_\la)^* \Ga \Phi \De (\xi) - T_\la \Ga^2 \Phi D_\De(\xi)
\end{equation}
for all $\xi \in \sD(D_\De)$. To verify these two identities, we use that $\sD(D \Phi) \su X$ is a core for the modular lift. The identity in Equation \eqref{eq:cruxI} then follows since it holds for all $\xi \in \sD(D \Phi)$ that
\[
\begin{split}
(D T_\la)^* \big( d_\Ga(G) G \Ga + \Ga G d_\Ga(G) \big) \Phi(\xi)
& = (D^2 T_\la)^* G \Ga G \Ga \Phi(\xi)  - (D T_\la)^* \Ga G \Ga G D \Phi(\xi) \\
& = (D^2 T_\la)^* \Phi \De^2(\xi) - (D T_\la)^* \Ga \Phi \De D_\De (\xi) .
\end{split}
\]
Moreover, the identity in Equation \eqref{eq:cruxII} follows from the computation
\[
T_\la d_\Ga(\Ga G) \Phi (\xi) = (D T_\la)^* \Ga G \Ga \phi(\xi) - T_\la \Ga^2 G D  \Phi(\xi)
=  (D T_\la)^* \Ga \Phi \De (\xi) - T_\la \Ga^2 \Phi D_\De(\xi) ,
\]
which is again valid for all $\xi \in \sD(D \Phi)$.

The identities in Equation \eqref{eq:cruxI} and Equation \eqref{eq:cruxII} imply that
\[
\begin{split}
& T_\la \big( \Phi \De^2 - \Ga^2 \Phi + d_\Ga(\Ga G) \Phi D_\De \big)(\xi)
+ (D T_\la)^* \big( d_\Ga(G) G \Ga + \Ga G d_\Ga(G) \big) \Phi(\xi) \\
& \q = T_\la(\Phi \De^2 - \Ga^2 \Phi)(\xi) + (D^2 T_\la)^* \Phi \De^2(\xi) - T_\la \Ga^2 \Phi (D_\De)^2(\xi) \\
& \q = \big( (1 + D^2) T_\la \big)^* \Phi \De^2 (\xi) - T_\la \Ga^2 \Phi (1 + (D_\De)^2)(\xi)
\end{split}
\]
for all $\xi \in \sD( (D_\De)^2)$. The result of the proposition then follows since
\[
\big( (1 + D^2) T_\la \big)^* \Phi \De^2 = \Phi \De^2 - T_\la \Ga^2 \Phi \cd \la \De^2 / r . \qedhere
\]
\end{proof}

\section{The unbounded Kasparov product}\label{s:unbkaspro}
Throughout this section we let $\sA$ and $\sB$ be $*$-algebras which satisfy the conditions in Assumption \ref{a:alg}. As usual we denote the $C^*$-completions by $A$ and $B$ and the operator space completions by $A_1$ and $B_1$. Furthermore, we fix a third $C^*$-algebra $C$.

On top of this data, we shall consider:
\begin{enumerate}
\item An unbounded modular cycle $(Y,D,\Ga)$ from $\sB$ to $C$ (with grading operator $\ga : Y \to Y$ in the even case).
\item A differentiable Hilbert $C^*$-module $X$ from $\sA$ to $\sB$ with a fixed differentiable generating sequence $\{ \xi_n\}_{n = 1}^\infty$.
\end{enumerate}

We let $\pi_A : A \to \sL(X)$ and $\pi_B : B \to \sL(Y)$ denote the $*$-homomorphisms associated with the above data. It will then be assumed that
\[
\pi_A(a) \in \sK(X) \q \T{for all } a \in A .
\]
%\medskip

To explain the aims of this section we form the interior tensor product $X \hot_B Y$ of Hilbert $C^*$-modules. We recall that this is the Hilbert $C^*$-module over $C$ defined as the completion of the algebraic tensor product of modules $X \ot_B Y$ with respect to the norm coming from the $C$-valued inner product
\[
\inn{ \cd , \cd } : X \ot_B Y \ti X \ot_B Y \to C \q \inn{x_0 \ot_B y_0, x_1 \ot_B y_1} := \inn{y_0, \pi_B(\inn{x_0,x_1})(y_1) } .
\]
The interior tensor product comes equipped with a $*$-homomorphism
\[
(\pi_A \ot 1) : A \to \sL(X \hot_B Y) \q (\pi_A \ot 1)(a) := \pi_A(a) \ot 1 .
\]
%\medskip

\emph{It is the main goal of this section to construct a new (and explicit) unbounded modular cycle from $\sA$ to $C$:
\[
X \hot_{\sB} (Y,D,\Ga) := (X \hot_B Y, D_\De, \De) .
\]
We shall refer to this new unbounded modular cycle as the unbounded Kasparov product of the differentiable Hilbert $C^*$-module $X$ and the unbounded modular cycle $(Y,D,\Ga)$.}
\medskip

Let us return to the interior tensor product $X \hot_B Y$. For each $\xi \in X$ we have a bounded adjointable operator
\[
T_\xi : Y \to X \hot_B Y \q y \mapsto \xi \ot_B y ,
\]
where the adjoint is given explicitly by
\[
T_\xi^* : X \hot_B Y \to Y \q x \ot_B y \mapsto \pi_B(\inn{\xi,x})(y) .
\]

For each $N \in \nn$ we may then define the bounded adjointable operator
\[
\Phi_N : X \hot_B Y \to \ell^2(\nn,Y) \q z \mapsto \sum_{n = 1}^N T_{\xi_n}^*(z) \de_n .
\]

\begin{lemma}\label{l:phistaII}
The sequence of bounded adjointable operators
\[
\big\{ \Phi_N \big\}_{N = 1}^\infty \q \Phi_N : X \hot_B Y \to \ell^2(\nn,Y)
\]
converges in operator norm to a bounded adjointable operator $\Phi : X \hot_B Y \to \ell^2(\nn,Y)$. Furthermore, we have that $\Phi^* : \ell^2(\nn,Y) \to X \hot_B Y$ has dense image.
\end{lemma}
\begin{proof}
Let us prove that the sequence $\{ \Phi_N \}$ is Cauchy in operator norm. To this end, we let $M > N$ be given and notice that
\[
\| \Phi_M - \Phi_N \|^2_\infty = \| \Phi_M^* - \Phi_N^* \|_\infty^2 = \| \Phi_M \Phi_M^* + \Phi_N \Phi_N^* - \Phi_N \Phi_M^* - \Phi_M \Phi_N^*\|_\infty .
\]
Furthermore, it may be verified that
\[
\Phi_M \Phi_M^* + \Phi_N \Phi_N^* - \Phi_N \Phi_M^* - \Phi_M \Phi_N^* 
= \sum_{n = N+1}^M \sum_{m = N+1}^M \pi_B(\inn{\xi_n,\xi_m}) \de_{nm} .
\]
Since the sequence
\[
\Big\{ \sum_{n,m = 1}^K \inn{\xi_n,\xi_m} \de_{nm} \Big\}_{K = 1}^\infty
\]
is a Cauchy sequence in $K(B_1)$ (by our assumption on the differentiable generating sequence $\{\xi_n\}$) and since the canonical map $B_1 \to B$ is completely bounded, this shows that $\{ \Phi_N \}$ is a Cauchy sequence as well.

To see that the image of $\Phi^* : \ell^2(\nn,Y) \to X \hot_B Y$ is dense it suffices (since $\{ \xi_n \}$ generates $X$ as a Hilbert $C^*$-module over $B$) to check that $\xi_n b \ot_B y \in \T{Im}(\Phi^*)$ for all $n \in \nn$, $b \in B$ and $y \in Y$. But this is clear since
\[
\Phi^*( \pi_B(b) (y) \cd \de_n) = \xi_n \ot_B \pi_B(b)(y) = \xi_n b \ot_B y .
\]
This ends the proof of the lemma.
\end{proof}

Let us recall from Subsection \ref{ss:stamodcyc} that the notation
\[
1 \ot D : \sD(1 \ot D) \to \ell^2(\nn,Y) \q \T{and} \q 1 \ot \Ga : \ell^2(\nn,Y) \to \ell^2(\nn,Y)
\]
refers to the diagonal operators induced by $D : \sD(D) \to Y$ and $\Ga : Y \to Y$.

The next lemma explains how $\Phi : X \hot_B Y \to \ell^2(\nn,Y)$ creates a link between the $*$-algebra $\sA$ and the unbounded modular cycle $(Y,D,\Ga)$. We recall from Proposition \ref{p:stamodcyc} that we have an unbounded modular cycle $\big( \ell^2(\nn,Y), 1 \ot D, 1 \ot \Ga \big)$ from $M(\sB)$ to $C$ (of the same parity as $(Y,D,\Ga)$).

\begin{lemma}\label{l:cheassu}
Let $a \in \sA$ and $\la \in \cc$ be given. Then the bounded adjointable operator $\Phi( \pi_A(a) \ot 1 + \la) \Phi^* : \ell^2(\nn,Y) \to \ell^2(\nn,Y)$ is differentiable with respect to the pair $(1 \ot D, 1 \ot \Ga)$. 

Furthermore, it holds for each $\ep \in (0,1]$ that the linear map $\si_\ep : \sA \to \sL(\ell^2(\nn,Y))$ defined by
\[
\si_\ep : a \mapsto (1 \ot \Ga + \ep)^{-1/2} d_{1 \ot \Ga}(\Phi (\pi_A(a) \ot 1) \Phi^*) (1 \ot \Ga + \ep)^{-1/2}
\]
is completely bounded with respect to the operator space norm $\| \cd \|_1$ on $\sA$ and we have that
\[
\sup_{\ep \in (0,1]}\| \si_\ep \|_{\T{cb}} < \infty .
\]
\end{lemma}
\begin{proof}
Let $\tau : \sA \to K(B_1)$ denote the completely bounded map defined by
\[
\tau(a) := \sum_{n,m = 1}^\infty \inn{\xi_n,\pi_A(a)(\xi_m)} \de_{nm} \q \T{for all } a \in \sA ,
\]
(where the complete boundedness is understood with respect to the operator space norm $\| \cd \|_1 $on $\sA$). Let also $g \in K(B_1)$ be given by $g := \sum_{n,m = 1}^\infty \inn{\xi_n,\xi_m} \de_{nm}$. Finally, as in Subsection \ref{ss:stamodcyc} we let $K(\pi_B) : K(B) \to \sL(\ell^2(\nn,Y))$ denote the $*$-homomorphism defined by
\[
K(\pi_B)\big( \sum_{n,m = 1}^\infty b_{nm} \de_{nm} \big)\big( \sum_{k = 1}^\infty y_k \de_k \big) := \sum_{n = 1}^\infty\Big( \sum_{m = 1}^\infty \pi_B(b_{nm})(y_m) \Big) \de_n .
\]
We then have the identity
\[
\Phi (\pi_A(a) \ot 1 + \la) \Phi^* = K(\pi_B)\big( \tau(a) + \la \cd g \big) : \ell^2(\nn,Y) \to \ell^2(\nn,Y) ,
\]
(where we are suppressing the inclusion $K(B_1) \to K(B)$). The fact that $\Phi (\pi_A(a) \ot 1 + \la) \Phi^*$ is differentiable with respect to $(\ell^2(\nn,Y), 1 \ot D, 1 \ot \Ga)$ is now a consequence of Proposition \ref{p:stamodcyc} (and the remarks following Definition \ref{d:unbkas}). Moreover, the remaining part of the lemma also follows from Proposition \ref{p:stamodcyc} by remarking that
\[
\si_\ep(a) = (\rho_{1 \ot \Ga,\ep} \ci \tau)(a) \q \T{for all } a \in \sA . \qedhere
\]
\end{proof}

It follows by Lemma \ref{l:phistaII} and Lemma \ref{l:cheassu} (with $a = 0$ and $\la = 1$) that $\Phi^* : \ell^2(\nn,Y) \to X \hot_B Y$ has dense image and that $\Phi \Phi^*$ is differentiable with respect to the pair $(1 \ot \Ga, 1 \ot D )$. In particular, we may form the modular lift
\[
(1 \ot D)_\De : \sD\big( (1 \ot D)_\De \big) \to X \hot_B Y ,
\]
as the closure of the symmetric unbounded operator $\Phi^* (1 \ot D) \Phi : \sD\big( (1 \ot D) \Phi \big) \to X \hot_B Y$. We also define the bounded adjointable operator 
\[
\De := \Phi^* (1 \ot \Ga) \Phi : X \hot_B Y \to X \hot_B Y .
\]

\begin{thm}\label{t:unbkas}
Suppose that the conditions outlined in the beginning of this section are satisfied. Then the triple $(X \hot_B Y, (1 \ot D)_\De, \De)$ is an unbounded modular cycle from $\sA$ to $C$. The parity of $(X \hot_B Y, (1 \ot D)_\De, \De)$ is the same as the parity of $(Y,D,\Ga)$ and the grading operator is given by $1 \ot \ga : X \hot_B Y \to X \hot_B Y$ in the even case.
\end{thm}
\begin{proof}
We verify each of the points in Definition \ref{d:unbkas} separately. 

The fact that $X \hot_B Y$ is a countably generated Hilbert $C^*$-module follows since both $X$ and $Y$ are countably generated by assumption.

The modular lift $(1 \ot D)_\De : \sD\big( (1 \ot D)_\De \big) \to X \hot_B Y$ is selfadjoint and regular by Theorem \ref{t:modselreg}.

The bounded operator $\De := \Phi^* (1 \ot \Ga) \Phi : X \hot_B Y \to X \hot_B Y$ is clearly positive and selfadjoint and it has dense image since $1 \ot \Ga : \ell^2(\nn,Y) \to \ell^2(\nn,Y)$ and $\Phi^* : \ell^2(\nn,Y) \to X \hot_B Y$ have dense images.

It is finally clear that the grading operator $1 \ot \ga : X \hot_B Y \to X \hot_B Y$ satisfies the constraints
\[
\begin{split}
& (1 \ot \ga) (\pi_A(a) \ot 1) = (\pi_A(a) \ot 1) (1 \ot \ga) \q
(1 \ot \ga) (1 \ot D)_\De = - (1 \ot D)_\De (1 \ot \ga) \\
& (1 \ot \ga) \De = \De (1 \ot \ga)
\end{split}
\]
for all $a \in A$ in the even case.

We now focus on the conditions $(1)$-$(4)$ in Definition \ref{d:unbkas}. 
\medskip

$(4)$: Let $a \in A$ be given. We need to show that $1/n (\De + 1/n)^{-1} (\pi_A(a) \ot 1) \to 0$ in the operator norm on $\sL(X \hot_B Y)$. To this end, we remark that there exists a positive and selfadjoint compact operator $K : X \to X$ with dense image such that $\Phi^* \Phi = K \ot 1 : X \hot_B Y \to X \hot_B Y$. (In fact we may choose $K := \sum_{n = 1}^\infty \te_{\xi_n, \xi_n}$). Since $\pi_A(a) \in \sK(X)$ we thus have that $\Phi^* \Phi (1/m + \Phi^* \Phi)^{-1} (\pi_A(a) \ot 1) \to \pi_A(a) \ot 1$ where the convergence takes place in operator norm. It therefore suffices to check that $1/n (\De + 1/n)^{-1} \Phi^* \Phi \to 0$ in operator norm. To prove this, we notice that $\Phi \Phi^* : \ell^2(\nn,Y) \to \ell^2(\nn,Y)$ lies in the image of the $*$-homomorphism $K(\pi_B) : K(B) \to \sL(\ell^2(\nn,Y))$. By Proposition \ref{p:stamodcyc} it thus follows that $\big( 1 \ot \Ga(\Ga + 1/m)^{-1}\big)\Phi \to \Phi$ in operator norm. We may therefore restrict our attention to showing that $1/n (\De + 1/n)^{-1} \Phi^* (1 \ot \Ga) \to 0$ in operator norm. But this is clear since $\De = \Phi^* (1 \ot \Ga) \Phi$ by definition.

$(1)$: Let again $a \in A$ be given. To verify that $(\pi_A(a) \ot 1) \cd \big( i + (1 \ot D)_\De \big)^{-1} : X \hot_B Y \to X \hot_B Y$ is a compact operator it suffices (by $(4)$ and Proposition \ref{p:compacres}) to check that $\Phi^* (1 \ot \Ga) \big(i + (1 \ot D) \big)^{-1} : \ell^2(\nn,Y) \to \ell^2(\nn,Y)$ is compact. But this is clear by Proposition \ref{p:stamodcyc} since $\Phi \Phi^* : \ell^2(\nn,Y) \to \ell^2(\nn,Y)$ lies in the image of $K(\pi_B) : K(B) \to \sL(\ell^2(\nn,Y))$ and since
\[
\big[ 1 \ot \Ga , ( i + (1 \ot D))^{-1} \big] = (i + (1 \ot D))^{-1} d_{1 \ot \Ga}(1) (i + (1 \ot D))^{-1} .
\]

$(2)$: Let $a \in \sA$ and $\la \in \cc$ be given. Let $z \in \sD\big( (1 \ot D) \Phi\big)$ be given (thus $\Phi(z) \in \sD( 1 \ot D)$). Then $\Phi (\pi_A(a) \ot 1 + \la) \De(z) = \Phi (\phi_A(a) \ot 1 + \la) \Phi^* (1 \ot \Ga) (\Phi z) \in \sD(1 \ot D)$ (by Lemma \ref{l:cheassu}) and therefore $(\pi_A(a) \ot 1 + \la) \De(z) \in \sD\big( (1\ot D) \Phi\big)$. Furthermore, we have that
\[
\begin{split}
\Phi^* (1 \ot D) \Phi (\pi_A(a) \ot 1 + \la) \De(z) & = \Phi^* (1 \ot \Ga) \Phi (\pi_A(a) \ot 1 + \la) \Phi^* (1 \ot D) \Phi(z) \\ 
& \q + \Phi^* d_{1 \ot \Ga}(\Phi (\pi_A(a) \ot 1 + \la) \Phi^*) \Phi(z) .
\end{split}
\]
But this implies that
\begin{equation}\label{eq:twicom}
\begin{split}
& \big((1 \ot D)_\De (\pi_A(a) \ot 1 + \la) \De - \De (\pi_A(a) \ot 1 + \la) (1 \ot D)_\De\big)(z) \\
& \q = \Phi^* d_{1 \ot \Ga}(\Phi (\pi_A(a) \ot 1 + \la) \Phi^*) \Phi(z) .
\end{split}
\end{equation}
Since $\sD\big( (1 \ot D) \Phi\big) \su X \hot_B Y$ is a core for the modular lift $(1 \ot D)_\De : \sD\big( (1 \ot D)_\De\big) \to X \hot_B Y$ this proves the relevant statement about twisted commutators.

$(3)$: Recall first that the linear map $\tau : \sA \to K(B_1)$, $a \mapsto \sum_{n,m = 1}^\infty \inn{\xi_n, \pi_A(a) \xi_m} \de_{nm}$ is assumed to be completely bounded.

Let now $a \in \sA$ and $\la \in \cc$ be given. Putting $g := \sum_{n,m = 1}^\infty \inn{\xi_n,\xi_m} \de_{nm} \in K(B_1)$ it follows by $(2)$ that $d_\De( a + \la ) = \Phi^* d_{1 \ot \Ga}(\tau(a) + \la \cd g) \Phi$. For each $\ep \in (0,1]$ and each $z \in X \hot_B Y$ we thus have that
\[
\begin{split}
& \big\| (\De + \ep)^{-1/2} d_\De(a + \la) (\De + \ep)^{-1/2}(z) \big\| \\
& \q = \lim_{\de \searrow 0} 
\big\| (\De + \ep)^{-1/2} \Phi^* (1 \ot \Ga)^{1/2} (1 \ot \Ga + \de)^{-1/2} d_{1 \ot \Ga}(\tau(a) + \la \cd g) \\
& \qqq \cd (1 \ot \Ga + \de)^{-1/2}  (1 \ot \Ga)^{1/2} \Phi (\De + \ep)^{-1/2}(z) \big\| \\
& \q \leq \| (\De + \ep)^{-1/2} \Phi^* (1 \ot \Ga)^{1/2} \|_\infty^2 \\ 
& \qqq \cd \sup_{\de \in (0,1]}\| \rho_{1 \ot \Ga,\de}(\tau(a) + \la \cd g) \|_\infty \cd \| z \| \\
& \q \leq  \sup_{\de \in (0,1]}\| \rho_{1 \ot \Ga,\de}(\tau(a) + \la \cd g) \|_\infty \cd \| z \| .
\end{split}
\]
Since we know that
\[
\sup_{\de \in (0,1]}\| \rho_{1 \ot \Ga,\de}(\tau(a) + \la \cd g) \|_\infty < \infty
\]
this proves the first part of the statement in $(3)$. The second part of the statement in $(3)$ follows by a similar argumentation. More precisely, we have the estimate:
\[
\sup_{\ep \in (0,1]}\| \rho_{\De,\ep} \|_{\T{cb}} \leq \sup_{\de \in (0,1]} \| \rho_{1 \ot \Ga,\de} \|_{\T{cb}} \cd \| \tau \|_{\T{cb}} . \qedhere
\]
\end{proof}

\section{The modular transform}\label{s:modtra}
Throughout this section we consider the following data:
\begin{enumerate}
\item an unbounded selfadjoint and regular operator $D : \sD(D) \to Y$ acting on a fixed Hilbert $C^*$-module $Y$;
\item a positive and selfadjoint bounded operator $\De : Y \to Y$ such that $\T{Im}(\De) \su Y$ is norm-dense.
\end{enumerate}

We make the following standing assumption:

\begin{assum}\label{modassu}
It is assumed that
\begin{enumerate}
\item the domain $\sD(D) \su Y$ is an invariant submodule for $\De : Y \to Y$ and the commutator
\[
D \De - \De D : \sD(D) \to Y
\]
is the restriction of a bounded adjointable operator $d(\De) : Y \to Y$;
\item the supremum of operator norms
\[
\sup_{\ep \in (0,1]}\| (\De + \ep)^{-1/2} d(\De) (\De + \ep)^{-1/2} \|_\infty
\]
is finite.
\end{enumerate}
In other words, we assume that the unit $1 \in \sL(Y)$ is differentiable with respect to the pair $(D,\De)$, see Definition \ref{d:unbkasII}.
\end{assum}

Let us choose a constant
\[
r \in ( \| \De \|_\infty^2 , \infty)
\]
and apply the notation:
\[
S_\la := (1 + \la \De^2 / r + D^2)^{-1} \q \T{and} \q R_\la := (1 + \la + D^2)^{-1}
\]
for all $\la \geq 0$.

We are interested in studying the \emph{modular transform} of the pair $(D, \De)$. This is the unbounded operator defined by
\[
G_{D,\De} : \De\big( \sD(D) \big) \to Y \q G_{D,\De}(\De \xi) 
:= \frac{1}{\pi} \int_0^\infty (\la r)^{-1/2} \De S_\la D( \De \xi) \, d \la
\]
for all $\xi \in \sD(D)$. In particular, we are interested in comparing the modular transform with the \emph{bounded transform} of $D : \sD(D) \to Y$. We recall that the bounded transform of $D$ is defined by $F_D := D (1 + D^2)^{-1/2} : Y \to Y$ and it follows that the bounded transform is a bounded extension of the unbounded operator
\[
F_D\big|_{\sD(D)} : \sD(D) \to Y \q \eta \mapsto \frac{1}{\pi} \int_0^\infty \la^{-1/2} R_\la D(\eta) \, d\la .
\]

The modular transform will play a key role in our later proof of one of the main theorems in this paper, namely that the bounded transform of an unbounded modular cycle yields a bounded Kasparov module and hence a class in $KK$-theory, see Theorem \ref{t:kasmod}.

We notice that the modular transform has been obtained from the bounded transform by making a non-commutative change of variables in the integral over the half-line. Indeed, the idea is just to replace the scalar-valued variable $\la \geq 0$ by the operator-valued variable $\la \cd \De^2/r$. In the case where $D$ and $\De$ actually commute it can therefore be proved that the modular transform is just a restriction of the bounded transform to $\De\big( \sD(D) \big) \su Y$. However, in the case of real interest, thus when $d(\De) \neq 0$, there is a substantial error-term appearing and a great deal of this section is devoted to controlling the size of this error-term. There are easier proofs of the main results of this section when the modular operator $\De : Y \to Y$ is assumed to be invertible (as a bounded operator). One of the important points of the whole theory that we are developing here is however that $\De : Y \to Y$ is allowed to have zero in the spectrum. This condition should therefore not be relaxed.

Before we go any further let us immediately notice that the integral
\[
\int_0^\infty (\la r)^{-1/2} \De S_\la D( \De \xi) \, d \la ,
\]
appearing in the definition of the modular transform, converges absolutely for all $\xi \in \sD(D)$. Indeed, this follows from the estimate
\begin{equation}\label{eq:prelmoduI}
\| \De S_\la D( \De \xi) \| \leq \| \De S_\la \De \|_\infty \cd \| D \xi \| + \| \De S_\la d(\De) \|_\infty \cd \| \xi \|
\end{equation}
and since Lemma \ref{l:eleestI} and Assumption \ref{modassu} $(2)$ imply that we may find a constant $C > 0$ such that
\begin{equation}\label{eq:prelmoduII}
\| \De S_\la \De \|_\infty \leq C \cd (1 + \la)^{-1} \q \T{and} \q \| \De S_\la d(\De) \|_\infty \leq C \cd (1 + \la)^{-3/4}
\end{equation}
for all $\la \geq 0$.

\subsection{Preliminary algebraic identities}\label{ss:prealgide}
Let us apply the notation
\[
K := 1 - \De^2/r \q \T{and} \q X_\la :=  \la \cd R_\la K \q \T{for all } \la \geq 0.
\]
%for all $$.

We start our work on understanding the modular transform
\[
G_{D,\De} : \De\big( \sD(D) \big) \to Y \q G_{D,\De}(\De \xi) = \frac{1}{\pi} \int_0^\infty (\la r)^{-1/2} \De S_\la D(\De \xi) \, d\la
\]
by rewriting the (modular) resolvent $S_\la = (1 + \la \De^2/r  + D^2)^{-1}$ in a way that is more amenable to a computation of the integral appearing in the expression for the modular transform. More precisely, we first expand the (modular) resolvent $S_\la : Y \to Y$ as a power-series involving the (standard) resolvent $R_\la : Y \to Y$ and the bounded adjointable operator $K : Y \to Y$. We then reorganize this power-series by moving all the $K$-terms to the left and all the $R_\la$-terms to the right and during this procedure we pick up an explicit error-term. These steps will be accomplished in the present subsection.

\begin{lemma}\label{l:firreside}
For each $\la \geq 0$ we have the identities
\[
S_\la  = \sum_{n = 0}^\infty X_\la^n \cd R_\la = (1 - X_\la)^{-1} R_\la = R_\la (1 - X_\la^*)^{-1} ,
\]
where the sum converges absolutely. % and where the square root is holomorphic on $\cc \sem (-\infty,0]$.
\end{lemma}
\begin{proof}
Let $\la \geq 0$ be given. By the resolvent identity we have that
\[
\begin{split}
R_\la - S_\la & = (\la + 1 + D^2)^{-1} - (\la \De^2/r + 1 + D^2)^{-1}  \\ 
& = - (\la + 1 + D^2)^{-1} \la (1 - \De^2/r) (\la \De^2/r + 1 + D^2)^{-1} \\
& =  - X_\la \cd S_\la .
\end{split}
\]
Since $\| \De^2 \|_\infty < r$ we have that $\| X_\la \|_\infty \leq \la (1 + \la)^{-1} < 1$. We may thus conclude that $1 - X_\la : Y \to Y$ is invertible with $(1 - X_\la)^{-1} = \sum_{n = 0}^\infty X_\la^n$, where the sum converges absolutely. From the above, we deduce that  
\[
S_\la = (1 - X_\la)^{-1} R_\la  = \sum_{n = 0}^\infty X_\la^n \cd R_\la .
\]
Since $S_\la^* = S_\la$ and $R_\la^* = R_\la$ we also see that $S_\la = R_\la (1 - X_\la^*)^{-1}$. This proves the lemma.
%
%To prove the last identity, we notice that the spectrum of $(1 - X_\la)^{-1}$ is included in $\cc \sem (-\infty, 0 ]$ and we thus have the identity $(1 - X_\la)^{-1} = (1 - X_\la)^{-1/2} (1 - X_\la)^{-1/2}$ (using a holomorphic branch of the square root on $\cc \sem (-\infty,0]$). The last identity then follows since $(1 - X_\la) R_\la = R_\la (1 - X_\la^*)$.
\end{proof}

We will from now on apply the notation 
\[
I(T) := [D^2,T] : \sD(D^2) \to Y
\]
whenever $T : Y \to Y$ is a bounded adjointable operator such that $T(\sD(D^2)) \su \sD(D^2)$.

\begin{lemma}\label{l:reorg}
Let $\la \geq 0$, $n \in \nn$ and $k \in \nn \cup \{0\}$ be given. We have the identity
\[
X_\la^n \cd \De^k = (X_\la^*)^{n-1} K \De^k R_\la \la - 
I \big( X_\la^n \De^k \big) R_\la .
\]
\end{lemma}
\begin{proof}
We compute that 
\[
I(X_\la^n \De^k) R_\la = \la K X_\la^{n-1} \De^k R_\la - X_\la^n \De^k
= \la (X_\la^*)^{n-1} K \De^k R_\la - X_\la^n \De^k ,
\]
where we are using that $K X_\la = X_\la^* K$. This proves the lemma.
%
%Let $m \in \nn$ and notice that
%\[
%K^m R_\la = R_\la K^m - [R_\la,K^m] = R_\la K^m + D R_\la d(K^m) R_\la + R_\la d(K^m) D R_\la
%\]
%The proof now follows by induction on $n \in \nn_0$.
\end{proof}

\begin{lemma}\label{l:reorgser}
Let $\la \geq 0$, $n \in \nn$ and $k \in \nn \cup \{0\}$ be given. We have the identity
\[
X_\la^n \cd \De^k = \De^k \la^n K^n R^n_\la - \sum_{j = 0}^{n-1} K^j \cd I\big( X_\la^{n-j} \De^k \big) R_\la^{j+1} \la^j .
\]
\end{lemma}
\begin{proof}
The proof runs by induction using the identity in Lemma \ref{l:reorg}.
\end{proof}

For each $m \in \nn$ and each $\la \geq 0$ we define the bounded adjointable operator
\[
L_\la(m) := I\big( (1 - X_\la^m) S_\la K \la \De^3 \big) R_\la : Y \to Y .
\]

\begin{lemma}\label{l:reorgserI}
Let $\la \geq 0$ and $N \in \nn$ be given. We have the identity 
\[
\sum_{n = 0}^N X_\la^n \cd \De^3 \cd R_\la = \sum_{n = 0}^N \De^3 \la^n K^n R_\la^{n+1}
- \sum_{n = 0}^{N-1} K^n \cd L_\la(N-n) \cd R_\la^{n + 1} \la^n .
\]
\end{lemma}
\begin{proof}
By an application of Lemma \ref{l:reorgser} (and a reordering of terms) we obtain that
\[
\begin{split}
& \sum_{n = 0}^N X_\la^n \cd \De^3 \cd R_\la \\
& \q = 
\sum_{n = 0}^N \De^3 \la^n K^n R^{n+1}_\la 
- \sum_{n = 1}^N \sum_{j = 0}^{n-1} K^j \cd I \big( X_\la^{n-j} \De^3\big) R_\la^{j+2} \la^j \\
& \q = \sum_{n = 0}^N \De^3 \la^n K^n R^{n+1}_\la  
- \sum_{j = 0}^{N-1} \sum_{m = 1}^{N-j} K^j \cd I \big( X_\la^m \De^3\big) R_\la^{j + 2} \la^j  .
\end{split}
\]
The result of the lemma now follows since Lemma \ref{l:firreside} implies that
\[
\sum_{m = 1}^{N- j} X_\la^m = \sum_{m = 0}^{N - j - 1} X_\la^m R_\la K \la = (1 - X_\la^{N - j}) S_\la K \la . \qedhere
\]
\end{proof}

For each $\la \geq 0$ we define the bounded adjointable operator
\[
L_\la := I( S_\la K \la \De^3) R_\la : Y \to Y .
\]

\begin{lemma}\label{l:conell}
Let $\la \geq 0$ be given. Then the sequence $\{ L_\la(m) \}_{m = 1}^\infty$ converges to $L_\la : Y \to Y$ in operator norm.
\end{lemma}
\begin{proof}
Using the Leibniz rule we see that it suffices to verify that the sequence $\big\{ I(X_\la^m) S_\la \big\}_{m = 1}^\infty$ converges to zero in operator norm. However, using the Leibniz rule one more time, we obtain that
\[
\begin{split}
I(X_\la^m) S_\la 
& = - \sum_{j = 0}^{m-1} X_\la^j I(R_\la \De^2) X_\la^{m-1 -j} S_\la \cd \la/r \\
& = - \sum_{j = 0}^{m-1} X_\la^j I(R_\la \De^2) S_\la (X_\la^*)^{m-1 -j} \cd \la/r .
\end{split}
\]
Remark here that Lemma \ref{l:firreside} indeed implies that $X_\la S_\la = S_\la X_\la^*$. The result of the lemma now follows easily by noting that $\| X_\la \|_\infty \leq \la(1 + \la)^{-1} < 1$. Indeed, we may then find a constant $C > 0$ such that $\| I(X_\la^m) S_\la \|_\infty \leq C \cd m \cd \big( \la (1 + \la)^{-1} \big)^{m-1}$ for all $m \in \nn$.
\end{proof}

We are now ready to prove the main result of this subsection. It provides an expansion of $S_\la \De^3 : Y \to Y$ where the first power-series appearing can be directly related (after integration over the half-line) to the bounded adjointable operator $(1 + D^2)^{-1/2} : Y \to Y$. The exponent $3$ that appears here (and earlier in this section) is not special, we only need that it is large enough for certain estimates to be valid later on.

\begin{prop}\label{p:prealgideI}
Let $\la \geq 0$ be given. Then we have the identity
\[
S_\la \De^3 = \sum_{n = 0}^\infty \De^3 K^n R_\la^{n + 1} \la^n
- \sum_{n = 0}^\infty K^n L_\la R_\la^{n + 1} \la^n - (1 - X_\la)^{-1} I(R_\la \De^3) R_\la ,
\]
where each of the sums converges absolutely in operator norm.
\end{prop}
\begin{proof}
It is clear that the sums converge absolutely in operator norm. Indeed, this follows since $\| K \|_\infty \leq 1$ and since $\| R_\la \cd \la \|_\infty \leq \la (\la + 1)^{-1} < 1$.

To continue, we notice that Lemma \ref{l:firreside} entails that
\[
\begin{split}
S_\la \De^3
& = (1 - X_\la)^{-1} \De^3 R_\la + (1 - X_\la)^{-1} [R_\la, \De^3] \\
& = (1 - X_\la)^{-1} \De^3 R_\la - (1 - X_\la)^{-1} I(R_\la \De^3) R_\la .
\end{split}
\]

Now, by an application of Lemma \ref{l:reorgserI}, we see that we may restrict our attention to proving that the sequence
\[
\big\{ \sum_{j = 0}^{N-1} K^j \cd L_\la(N - j) \cd R_\la^{j + 1} \la^j \big\}_{N = 1}^\infty
\]
converges in operator norm to $\sum_{j = 0}^\infty K^j L_\la R_\la^{j + 1} \la^j$. To this end, we define
\[
C_0 := \sup_{n \in \nn} \| L_\la(n) \|_\infty \q \T{and} \q C_1 := \sum_{j = 0}^\infty \| R_\la^{j + 1} \la^j \|_\infty .
\]
Both of these constants are of course finite. Let now $\ep > 0$ be given. By Lemma \ref{l:conell} we may choose $N_0, M_0 \in \nn$ such that 
\[
\begin{split}
& \| L_\la - L_\la(n) \|_\infty < \frac{\ep}{3 (C_1 + 1)} \q \T{for all } n \geq N_0 \q \T{and} \\
& \sum_{j = M_0 + 1}^\infty \| R_\la^{j + 1} \la^j \|_\infty < \frac{\ep}{3 (C_0 + 1)} .
\end{split}
\]
It is then straightforward to verify that
\[
\big\| \sum_{j = 0}^\infty K^j L_\la R_\la^{j + 1} \la^j - \sum_{j = 0}^{N-1} K^j L_\la(N - j) R_\la^{j+1} \la^j \big\|_\infty < \ep
\]
for all $N \geq N_0 + M_0$. This proves the present proposition.
\end{proof}

\subsection{Integral formulae for the square root}
The aim of this subsection is to compute the integral over the half-line of the continuous map
\[
f : (0,\infty) \to \sL(Y) \q f : \la \mapsto (\la r)^{-1/2} \sum_{n = 0}^\infty \De^6 K^n R_\la^{n+1} \la^n
\]
which appears (up to a factor of $(\la r)^{-1/2}$) in the expression for $\De^3 S_\la \De^3 : Y \to Y$ obtained in Proposition \ref{p:prealgideI}. The main result of this subsection is the explicit formula
\[
\frac{1}{\pi} \int_0^\infty f(\la) \, d\la = \De^5 (1 + D^2)^{-1/2} ,
\]
which is proved in Proposition \ref{p:comboutraI}.

We start by recalling a general result on integral formulae for powers of resolvents:

\begin{lemma}\label{l:betide}
Let $\La : \sD(\La) \to Y$ be an unbounded selfadjoint and regular operator and let $p,q > 0$. We have the identity
\[
B(p,q) \cd (1 + \La^2)^{-q} = \int_0^\infty \la^{p-1} (1 + \la + \La^2)^{-p-q} \, d \la ,
\]
where the integral converges absolutely and where
\[
B(p,q) := \int_0^\infty \mu^{p-1} (1 + \mu)^{-p-q} \, d \mu
\]
is the beta-function.
\end{lemma}
\begin{proof}
Notice that a change of variables ($\la = \mu \cd t$) implies that
\[
\int_0^\infty \la^{p-1} (\la + t)^{-p-q} \, d\la = t^{-q}\cd \int_0^\infty \mu^{p-1} (1 + \mu)^{-p-q} \, d\mu
\]
for all $t > 0$. The result now follows by an application of the functional calculus for unbounded selfadjoint and regular operators, see \cite{Wor:UAQ,WoNa:OTC}.
\end{proof}

Let us fix two elements $\xi,\eta \in Y$ together with a state $\rho : B \to \cc$ on the base $C^*$-algebra. %We will often apply the notation
%\[
%\inn{y_0, y_1}_\rho := \rho\big( \inn{y_0,y_1} \big) \q y_0 , y_1 \in Y
%\]
%for the localized inner product. %(with respect to the state $\rho : B \to \cc$).

The next lemma reduces the computation of the integral $\frac{1}{\pi} \int_0^\infty f(\la) \, d\la$ to a (delicate) matter of interchanging an infinite sum and an integral.

\begin{lemma}\label{l:weasqr}
The sequence of partial sums
\[
\Big\{ \frac{1}{\pi} \sum_{n = 0}^N \int_0^\infty (\la r)^{-1/2} \rho\big( \binn{\De^2 K^n R_\la^{n+1} \la^n \xi,\eta} \big) \, d\la \Big\}_{N = 1}^\infty
\]
converges to $\rho\big( \binn{ \De (1 + D^2)^{-1/2}\xi, \eta} \big)$.
\end{lemma}
\begin{proof}
By Lemma \ref{l:betide} we have that
\[
\begin{split}
& \frac{1}{\pi \sqrt{r}} \sum_{n = 0}^N \int_0^\infty \la^{-1/2} \cd \De^2 K^n \la^n R_\la^{n+1} \, d\la \\
& \q = \frac{1}{\pi \sqrt{r}} \sum_{n = 0}^N \De^2  K^n \cd \int_0^\infty \la^{n-1/2} \cd R_\la^{n+1} \, d\la \\
& \q = \frac{1}{\pi \sqrt{r}} \sum_{n = 0}^N \De^2 K^n \cd B(n + 1/2,1/2) \cd (1 + D^2)^{-1/2}
\end{split}
\]
for all $N \in \nn$. Thus, it suffices to check that the sequence of complex numbers
\[
\Big\{ \frac{1}{\pi \sqrt{r}} \sum_{n = 0}^N \rho\big( \binn{ \De^2 K^n \xi, \eta} \big)  \cd B(n + 1/2,1/2) \Big\}_{N = 1}^\infty
\]
converges to $\rho\big( \inn{\De \xi,\eta}\big) \in \cc$.

Because of the polarization identity, we may assume (without loss of generality) that $\xi = \eta$.

Let now $\mu > 0$ be fixed and notice that
\[
\begin{split}
\sum_{n = 0}^\infty \De^2 K^n \cd (1 + \mu)^{-n-1} & = \De^2 (1 + \mu)^{-1} \cd \sum_{n = 0}^\infty \big( K (1 + \mu)^{-1}\big)^n \\
& = \De^2 (1 + \mu)^{-1} \big( 1 - K(1 + \mu)^{-1} \big)^{-1} \\ 
& = \De^2 (1 - K + \mu)^{-1} = \De^2 (\De^2/r + \mu)^{-1} .
\end{split}
\]
Next, by a change of variables $(\mu = \la \cd \De^2 /r )$, we obtain that
\[
\frac{1}{\pi} \int_0^\infty (\mu r)^{-1/2} \De^2 (\De^2/r + \mu)^{-1} \, d\mu
= \frac{1}{\pi} \int_0^\infty \la^{-1/2} (1 + \la)^{-1} \De \, d\la = \De .
\]
Therefore, by the monotone convergence theorem, we may conclude that
\[
\begin{split}
& \frac{1}{\pi \sqrt{r}} \cd \lim_{N \to \infty} \Big( \sum_{n = 0}^N \rho( \inn{ \De^2 K^n \xi, \xi } ) \cd B(n + 1/2,1/2)  \Big) \\
& \q = \frac{1}{\pi} \cd \lim_{N \to \infty} \Big( \int_0^\infty (\mu r)^{-1/2} \sum_{n = 0}^N \rho\big( \binn{ \De^2 K^n (1 + \mu)^{-n-1} \xi, \xi } \big) \, d\mu \Big) \\
& \q = \frac{1}{\pi} \int_0^\infty (\mu r)^{-1/2} \rho\big( \binn{ \De^2(\De^2/r + \mu)^{-1} \xi, \xi} \big) \, d\mu \\
& \q = \rho( \inn{\De \xi,\xi} ) .
\end{split}
\]
This proves the lemma.
\end{proof}

In order to compute the integral of $f : (0,\infty) \to \sL(Y)$ (and to show that this function is integrable) we now want to apply the Lebesgue dominated convergence theorem. Or in other words we need to find a positive integrable function $g : (0,\infty) \to [0,\infty)$ such that
\[
(\la r)^{-1/2} \big\| \sum_{n = 0}^N \De^6 K^n R_\la^{n+1} \la^n \big\|_\infty \leq g(\la) \q \T{for all } \la > 0 \, , \, \, N \in \nn .
\]
This turns out to be a subtle problem and the solution will rely on the identities derived in Subsection \ref{ss:prealgide} and the estimates that we carry out in the appendix to this paper. On top of these estimates we need the following two lemmas:

\begin{lemma}\label{l:intrig}
Let $p \in (-\infty,2]$ be given. We have the identity
\[
\sum_{n = 0}^\infty (1 + D^2)^p R_\la^{2n+2} \la^{2n} = (1 + D^2)^{p-1}(2 \la + 1 + D^2)^{-1} ,
\]
where the sum converges absolutely in operator norm for all $\la \geq 0$.
\end{lemma}
\begin{proof}
It is clear the the sum converges absolutely for all $\la \geq 0$. To prove the relevant identity we let $\la \geq 0$ be given and compute as follows:
\[
\begin{split}
\sum_{n = 0}^\infty (1 + D^2)^p R_\la^{2n+2} \la^{2n} 
& = (1 + D^2)^p R_\la^2 (1 - R_\la^2 \la^2)^{-1} \\
& = (1 + D^2)^p (1 + D^2)^{-1}(2 \la + 1 + D^2)^{-1} . \qedhere
\end{split} 
\] 
\end{proof}
%
%Recall from Subsection \ref{ss:regula} that $Y_\rho$ denotes the Hilbert space obtained by separation and completion of $Y$ with respect to the (pre) inner product $\inn{\cd,\cd}_\rho := \rho\big( \inn{\cd,\cd} \big)$. Furthermore, we have the notation $[\cd] : Y \to Y_\rho$ for the quotient map.

\begin{lemma}\label{l:conkay}
The sequence of partial sums
\[
\big\{ \sum_{n = 0}^N \De^2 K^{2n} \big\}_{N = 0}^\infty
\]
is bounded in operator norm.
\end{lemma}
\begin{proof}
This follows from the identities
\[
\begin{split}
\sum_{n = 0}^N (\De^2/r) K^{2n} (2 - \De^2/r) 
= \sum_{n = 0}^N (1 - K^2) K^{2n} = 1 - K^{2(N + 1)}
\end{split}
\]
by noting that $2 - \De^2/r : Y \to Y$ is invertible and that $\| K  \|_\infty \leq 1$.
%
%Let $\ze := (2 - \De^2/r)^{-1} \eta$. It then suffices to check that $[ (1 - \De^2/r)^{2(N+1)} \ze ] \to 0$. But this follows since the localized operator $\De^2/r \ot 1 : Y_\rho \to Y_\rho$ is a positive bounded operator with dense image and with $\| \De^2/r \ot 1 \|_\rho \leq \| \De^2/r\| \leq 1$, see \cite{}.
\end{proof}

\begin{lemma}\label{l:comboutraI}
There exists a positive integrable function $g : (0,\infty) \to [0,\infty)$ such that
\[
(\la r)^{-1/2} \big\| \sum_{n = 0}^N \De^6 K^n R_\la^{n+1} \la^n \big\|_\infty \leq g(\la) 
\]
for all $\la \in (0,\infty)$ and all $N \in \nn$.
\end{lemma}
\begin{proof}
By an application of Lemma \ref{l:reorgserI} we obtain that
\begin{equation}\label{eq:comI}
\begin{split}
\sum_{n = 0}^N \De^6 K^n R_\la^{n+1} \la^n & = \sum_{n = 0}^N \De^3 X_\la^n R_\la \De^3
+ \sum_{n = 0}^N \De^3 X_\la^n I(R_\la \De^3) R_\la \\
& \qqq + \sum_{n = 0}^{N-1} \De^3 K^n L_\la(N-n) R_\la^{n + 1} \la^n
\end{split}
\end{equation}
for all $\la \geq 0$ and all $N \in \nn$. We estimate the operator norm of each of these terms separately.

For the first term in Equation \eqref{eq:comI} we apply Lemma \ref{l:eleestI} to obtain that
\[
\big\| \sum_{n = 0}^N \De^3 X_\la^n R_\la \De^3 \big\|_\infty \leq \| \De^3 S_\la \De^3 \|_\infty \leq r^3 (1 + \la)^{-1}
\]
for all $\la \geq 0$ and all $N \in \nn$.

For the second term in Equation \eqref{eq:comI} we apply Lemma \ref{l:eleestI} and Lemma \ref{l:eleome} to find a constant $C_1 > 0$ such that
\[
\begin{split}
& \| \De^3 \sum_{n = 0}^N X_\la^n I(R_\la \De^3) R_\la \|_\infty \\
& \q \leq \| \De^3 \sum_{n = 0}^N X_\la^n D R_\la d(\De^3) R_\la \|_\infty
+ \| \De^3 \sum_{n = 0}^N X_\la^n R_\la d(\De^3) D R_\la \|_\infty \\
& \q =  \| \De^3 (1 - X_\la^{N+1}) (D S_\la )^* \cd d(\De^3) R_\la \|_\infty
+ \| \De^3 (1 - X_\la^{N + 1}) S_\la d(\De^3) D R_\la \|_\infty \\
& \q \leq C_1 \cd (1 + \la)^{-3/4}
\end{split}
\]
for all $\la \geq 0$ and all $N \in \nn$. We recall here that $\sup_{\ep \in (0,1]}\| (\De + \ep)^{-1/2}d(\De^3) (\De + \ep)^{-1/2} \|_\infty < \infty$ by Assumption \ref{modassu} and that the sequence $\{ \De^{1/2}(\De + 1/m)^{-1/2} \}_{m = 1}^\infty$ converges strictly to the identity operator on $Y$.

For the third term in Equation \eqref{eq:comI} we apply the Cauchy-Schwarz inequality to obtain that 
\begin{equation}\label{eq:estI}
\begin{split}
& \big\| \sum_{n = 0}^{N-1} \De^3 K^n L_\la(N-n) R_\la^{n+1} \la^n \big\|_\infty \\
& \q \leq \big\| \sum_{n = 0}^{N-1} \De^3 K^n L_\la(N - n) (1 + D^2)^{-1} L_\la(N - n)^* K^n \De^3 \big\|^{1/2}_\infty  \\
& \qqq \cd \big\| \sum_{n = 0}^{N - 1} (1 + D^2) R_\la^{2n+2} \la^{2n} \big\|^{1/2}_\infty .
\end{split}
\end{equation}
We then note that Proposition \ref{p:truerr} and Lemma \ref{l:conkay} imply that there exists a constant $C_2 > 0$ such that
\[
\big\| \sum_{n = 0}^{N-1} \De^3 K^n L_\la(N - n) (1 + D^2)^{-1} L_\la(N - n)^* K^n \De^3 \big\|^{1/2}_\infty
\leq C_2 \cd (1 + \la)^{-1/8}
\]
for all $N \in \nn$ and all $\la \geq 0$. Furthermore, by Lemma \ref{l:intrig} we have that
\[
\big\| \sum_{n = 0}^{N - 1} (1 + D^2) R_\la^{2n+2} \la^{2n} \big\|^{1/2}_\infty
\leq (2 \la + 1)^{-1/2}
\]
for all $N \in \nn$ and all $\la \geq 0$.  This provides an adequate norm estimate of the final term in Equation \eqref{eq:comI} and the lemma is therefore proved.
\end{proof}

The main result of this subsection now follows by Lemma \ref{l:weasqr}, Lemma \ref{l:comboutraI}, and the Lebesgue dominated convergence theorem:

\begin{prop}\label{p:comboutraI}
The continuous function
\[
f : (0,\infty) \to \sL(Y) \q f(\la) := (\la r)^{-1/2} \sum_{n = 0}^\infty \De^6 K^n R_\la^{n+1} \la^n
\]
is absolutely integrable (with respect to Lebesgue measure on $(0,\infty)$ and the operator norm). Furthermore, the integral is given explicitly by
\[
\frac{1}{\pi} \int_0^\infty f(\la) \, d\la = \De^5 (1 + D^2)^{-1/2} .
\]
\end{prop}

\subsection{Comparison with the bounded transform}
We are now ready to prove the main theorem of this section, which, at least in practice, says that the bounded transform $F_D = D(1 + D^2)^{-1/2}$ has the same compactness properties as the modular transform
\[
G_{D,\De} : \De \xi \mapsto \frac{1}{\pi} \int_0^\infty (\la r)^{-1/2} \De S_\la D (\De \xi) \, d\la \q \xi \in \sD(D)
\]
after multiplication with a sufficiently large power of the modular operator $\De$. We remark that it might be possible to improve the exponent $p \in [0,1/4)$ appearing in the main theorem here below. Allowing exponents $p \in [1/4,1/2)$ could be important in a more detailed analysis of summability properties in relation to the unbounded Kasparov product. In the present text we limit ourselves to the question of compactness of resolvents and defer a deeper analysis of decay properties of eigenvalues to future work.

%We start by comparing the bounded transform to a bounded operator, which is 

\begin{prop}\label{t:comboutraI}
Suppose that the conditions in Assumption \ref{modassu} are satisfied and let $p \in [0,1/2)$ be given. Then the difference of unbounded operators
\[
\begin{split}
& \De^5 F_D \cd (1 + D^2)^p \\
& \q - \frac{1}{\pi} \int_0^\infty (\la r)^{-1/2} \De^3 (1 - X_\la)^{-1} \De^3 R_\la \, d\la \cd D (1 + D^2)^p : \sD(|D|^{2p + 1}) \to Y
\end{split}
\]
has a bounded extension to $Y$.
\end{prop}
\begin{proof}
It follows from Proposition \ref{p:prealgideI} that
\[
\begin{split}
\De^3 (1 - X_\la)^{-1} \De^3 R_\la 
& = \De^3 S_\la \De^3 + \De^3 (1- X_\la)^{-1} I( R_\la \De^3) R_\la \\
& = \sum_{n = 0}^\infty \De^6 K^n R_\la^{n + 1} \la^n - \sum_{n = 0}^\infty \De^3 K^n L_\la R_\la^{n + 1} \la^n 
\end{split}
\]
and hence, by an application of Proposition \ref{p:comboutraI}, we may focus our attention on proving that the unbounded operator
\begin{equation}\label{eq:comII}
\begin{split}
& \frac{1}{\pi} \int_0^\infty (\la r)^{-1/2} \sum_{n = 0}^\infty \De^3 K^n L_\la R_\la^{n+1} \la^n \, d\la \cd D(1 + D^2)^p : \sD(|D|^{2p + 1}) \to Y
\end{split}
\end{equation}
has a bounded extension to $Y$.

To this end, we apply the Cauchy-Schwarz inequality to obtain that
\[
\begin{split}
& \big\| \sum_{n = 0}^\infty \De^3 K^n L_\la D (1 + D^2)^p R_\la^{n+1} \la^n \big\|_\infty \\
& \q \leq \sup_{N \in \nn} \big\| \sum_{n = 0}^N \De^3 K^n L_\la L_\la^* K^n \De^3 \big\|^{1/2}_\infty
\cd \big\| \sum_{n = 0}^\infty D^2 (1 + D^2)^{2p} R_\la^{2n+2} \la^{2n} \big\|^{1/2}_\infty
\end{split}
\]
for all $\la \geq 0$. Next, by an application of Proposition \ref{p:esterr} and Lemma \ref{l:conkay} we may find a constant $C_1 > 0$ such that
\[
\sup_{N \in \nn} \big\| \sum_{n = 0}^N \De^3 K^n L_\la L_\la^* K^n \De^3 \big\|^{1/2}_\infty \leq C_1 \cd (1 + \la)^{-1/2} .
\]
Furthermore, by Lemma \ref{l:intrig} we have that
\[
\big\| \sum_{n = 0}^\infty D^2 (1 + D^2)^{2p} R_\la^{2n+2} \la^{2n} \big\|^{1/2}_\infty \leq \big\| (1 + D^2)^{2p} (2 \la + 1 + D^2)^{-1} \big\|^{1/2}_\infty 
\leq (2 \la + 1)^{p - 1/2} .
\]
These estimates imply that the integral
\[
\frac{1}{\pi} \int_0^\infty (\la r)^{-1/2} \sum_{n = 0}^\infty \De^3 K^n L_\la \cd D(1 + D^2)^p R_\la^{n+1} \la^n \, d\la
\]
converges absolutely in operator norm and the proposition is therefore proved.
\end{proof}

Before proving our main theorem we present a few extra preliminary results on the modular transform:

\begin{lemma}\label{l:modadjII}
The modular transform $G_{D,\De} : \De( \sD(D)) \to Y$ has a densely defined adjoint $G_{D,\De}^* : \sD( G_{D,\De}^*) \to Y$. In fact, it holds that $\De( \sD(D)) \su \sD(G_{D,\De}^*)$ and that
\[
G_{D,\De}^*( \De \xi) =  \frac{1}{\pi} \int_0^\infty (\la r)^{-1/2} D S_\la \De (\De \xi) \, d \la  \q \xi \in \sD(D) ,
\]
where the integral converges absolutely in the norm on $Y$. 
%Moreover, the difference $G_{D,\De} - G_{D,\De}^* : \De( \sD(D)) \to Y$ has a bounded adjointable extension to $Y$.
\end{lemma}
\begin{proof}
It suffices to show that the integral
\[
\frac{1}{\pi} \int_0^\infty (\la r)^{-1/2} D S_\la \De^2 \xi \, d\la
\]
converges absolutely for all $\xi \in \sD(D)$. To this end, we compute that
\begin{equation}\label{eq:intadj}
\begin{split}
D S_\la \De^2 \xi & = \De D S_\la \De \xi + [D S_\la , \De] \De \xi
= \De S_\la D \De \xi + \De [D,S_\la] \De \xi + [D S_\la,\De] \De \xi  \\
& = \De S_\la D \De \xi - \la \cd \De S_\la d(\De^2 /r) S_\la \De \xi + 
d(\De) S_\la \De \xi + D [S_\la, \De] \De \xi .
\end{split}
\end{equation}
It follows from the computations in the beginning of Section \ref{s:modtra} (Equation \eqref{eq:prelmoduI} and \eqref{eq:prelmoduII}) that there exists a constant $C_1 > 0$ such that
\[
\| \De S_\la D \De \xi \| \leq C_1 \cd (1 + \la)^{-3/4} \q \T{for all } \la \geq 0 . 
\]
For the remaining three terms on the right hand side of Equation \eqref{eq:intadj} we may estimate the operator norm directly: for the first two of these three remaining terms we obtain from Lemma \ref{l:eleestI} that
\[
\begin{split}
& \| \la \cd \De S_\la d(\De^2 /r) S_\la \De \|_\infty \leq C_2 \cd (1 + \la)^{-3/4} \q \T{and} \\
& \| d(\De) S_\la \De \|_\infty \leq C_3 \cd (1 + \la)^{-3/4} \\
\end{split}
\]
for some constants $C_2,C_3 > 0$, which are independent of $\la \geq 0$. For the last of our three remaining terms we use Lemma \ref{l:eleome} and Lemma \ref{l:algide} to see that
\[
\| D [\De, S_\la ] \De \|_\infty \leq \| \Om_\la^* \Om_\la d(\De) S_\la \De \|_\infty + \| D S_\la d(\De) D S_\la \De \|_\infty
\leq C_4 \cd (1 + \la)^{-3/4} 
\]
for some constant $C_4 > 0$, which is again independent of $\la \geq 0$. These estimates imply the desired convergence result and the lemma is therefore proved.
\end{proof}

\begin{lemma}\label{l:moddom}
We have the inclusion
\[
\sD(D) \su \sD\big( \ov{ \De G_{D,\De} } \big)
\]
and it holds that
\begin{equation}\label{eq:moddomcon}
\ov{\De G_{D,\De}}( \xi) =  \frac{1}{\pi} \int_0^\infty (\la r)^{-1/2} \De^2 S_\la D \xi \, d\la \q \xi \in \sD(D),
\end{equation}
where the integral converges absolutely in the norm on $Y$.
\end{lemma}
\begin{proof}
Remark that Lemma \ref{l:modadjII} implies that $\De G_{D,\De} : \De( \sD(D)) \to Y$ is indeed closable. In fact, $\De G_{D,\De}$ has a densely defined adjoint since $\sD( D) \su \sD( (\De G_{D,\De})^* )$. Remark also that the integral on the right hand side of Equation \eqref{eq:moddomcon} converges absolutely since it follows by the computations in the proof of Lemma \ref{l:eleestII} that
\[
\| \De^2 S_\la D \xi \| \leq \| \De^2 S_\la \|_\infty \cd \| D \xi \| \leq C \cd (1 + \la)^{-3/4} \cd \| D \xi \|
\]
for some constant $C > 0$, which is independent of $\la \geq 0$.

Let now $\xi \in \sD(D)$ be given. It holds that $\De (1/n + \De)^{-1} \xi \to \xi$ and that $\De (1/n + \De)^{-1} \xi \in \sD( \De G_{D,\De})$ for all $n \in \nn$. It therefore suffices to show that
\[
\De G_{D,\De} \De (1/n + \De)^{-1} \xi \to \frac{1}{\pi} \int_0^\infty (\la r)^{-1/2} \De^2 S_\la D \xi \, d\la .
\]
For each $n \in \nn$, we have that
\[
\De G_{D,\De} \De (1/n + \De)^{-1} \xi
%= \frac{1}{\pi} \int_0^\infty (\la r)^{-1/2} \De^2 S_\la D \De (1/n + \De)^{-1} \xi \, d\la
= \frac{1}{\pi} \int_0^\infty (\la r)^{-1/2} \De^2 S_\la  D \De (1/n + \De)^{-1} \xi \, d\la .
\]
Now, since $\frac{1}{\pi} \int_0^\infty (\la r)^{-1/2} \De^2 S_\la \, d\la$ is a bounded adjointable operator we only need to prove that
\[
D \De (1/n + \De)^{-1} \xi \to D \xi .
\]
But this follows from the computations presented in the proof of Proposition \ref{p:modseladj} and Lemma \ref{l:strlimzer}. Indeed, one has to verify that
\[
d( \De (1/n + \De)^{-1}) = 1/n (1/n + \De)^{-1} d(\De) (1/n + \De)^{-1} \xi \to 0 . \qedhere
\]
\end{proof}

\begin{thm}\label{t:comboutraII}
Suppose that the conditions in Assumption \ref{modassu} are satisfied and let $p \in [0,1/4)$ be given. Then the difference
\[
\ov{ \De^5 G_{D,\De} } (1 + D^2)^p - \De^5 F_D (1 + D^2)^p : \sD( |D|^{2p + 1}) \to Y
\]
extends to a bounded adjointable operator on $Y$. In particular, we have that $\ov{\De^5 G_{D,\De}}$ is a bounded adjointable operator.
\end{thm}
\begin{proof}
By Lemma \ref{l:modadjII}, Lemma \ref{l:moddom} and Proposition \ref{t:comboutraI} we may focus on showing that
\[
\begin{split}
& \ov{\De^5 G_{D,\De}} \cd (1 + D^2)^p - \frac{1}{\pi} \int_0^\infty (\la r)^{-1/2} \De^3 (1 - X_\la)^{-1} \De^3 R_\la \, d \la \cd D(1 + D^2)^p \\
& \q = \frac{1}{\pi} \int_0^\infty (\la r)^{-1/2} \De^3 [ \De^3, (1 - X_\la)^{-1}] R_\la \, d\la \cd D(1 + D^2)^p
\end{split}
\]
extends from $\sD( |D|^{2p+1})$ to a bounded operator on $Y$. We achieve this by proving that the integral
\[
\frac{1}{\pi} \int_0^\infty (\la r)^{-1/2} \De^3 [ \De^3, (1 - X_\la)^{-1}] D(1 + D^2)^p R_\la \, d\la 
\]
converges absolutely in operator norm.

To this end, we start by computing that
\[
\begin{split}
[\De^3 , (1 - X_\la)^{-1}] 
& = (1 - X_\la)^{-1} [\De^3, X_\la] (1 - X_\la)^{-1} \\
& = \la (1 - X_\la)^{-1} [\De^3, R_\la] K (1 - X_\la)^{-1} \\
& = \la (1 - X_\la)^{-1} D R_\la d(\De^3) R_\la K (1 - X_\la)^{-1} \\
& \q + \la (1 - X_\la)^{-1} R_\la d(\De^3) D R_\la K (1 - X_\la)^{-1} \\
& = \la S_\la^{1/2} \Om_\la d(\De^3) S_\la K + \la S_\la d(\De^3) D S_\la K ,
\end{split}
\]
where we are using the notation from Lemma \ref{l:eleome}. We thus have that
\[
\begin{split}
\De^3[\De^3, (1 - X_\la)^{-1}] D & \su 
\la \De^3 S_\la^{1/2} \Om_\la d(\De^3) S_\la^{1/2} \Om_\la K 
+ \la \De^3 S_\la d(\De^3) \Om_\la^* \Om_\la K \\
& \q + \la \De^3 S_\la^{1/2} \Om_\la d(\De^3) S_\la d(\De^2/r) + \la \De^3 S_\la d(\De^3) S_\la d(\De^2/r) .
\end{split}
\]
The estimates in Lemma \ref{l:eleome}, Lemma \ref{l:eleestI}, and Lemma \ref{l:eleestII} then imply that there exists a constant $C > 0$ such that
\[
\| \De^3 [\De^3, (1 - X_\la)^{-1}] D \xi \|
\leq C \cd \la (1 + \la)^{-3/4} \cd \| \xi \|
\]
for all $\la \geq 0$ and all $\xi \in \sD(D)$. Since $\| (1 + D^2)^p R_\la \|_\infty \leq (1 + \la)^{-1 + p}$ for all $\la \geq 0$, we conclude that
\[
\| \De^3 [ \De^3, (1 - X_\la)^{-1}] D(1 + D^2)^p R_\la \|_\infty \leq C \cd (1 + \la)^{-3/4 + p}
\]
for all $\la \geq 0$. This proves the theorem since $p \in [0,1/4)$ by assumption.
\end{proof}

\section{The Kasparov module of an unbounded modular cycle}
Throughout this section we let $\sA$ be a $*$-algebra which satisfies the conditions of Assumption \ref{a:alg}. We then consider a fixed unbounded modular cycle $(X,D,\De)$ from $\sA$ to an arbitrary $C^*$-algebra $B$. We assume that $(X,D,\De)$ is either of even or odd parity and in the even case we denote the $\zz/2\zz$-grading operator by $\ga : X \to X$. We apply the notation
\[
F_D := D(1 + D^2)^{-1/2}
\]
for the \emph{bounded transform} of the unbounded selfadjoint and regular operator $D : \sD(D) \to X$.
\medskip

\emph{The aim of this section is to show that the pair $(X,F_D)$ is a bounded Kasparov module from $A$ to $B$ and hence that our unbounded modular cycle gives rise to a class in the $KK$-group, $KK_p(A,B)$ (where $p = 0,1$ according to the parity of $(X,D,\De)$).}
\medskip

We will thus prove (see Theorem \ref{t:kasmod}) that the following holds for all $a \in A$:
\begin{enumerate}
\item $\pi(a)(F_D^2 - 1) \in \sK(X)$;
\item $\pi(a)(F_D - F_D^*) \in \sK(X)$;
\item $[F_D,\pi(a)] \in \sK(X)$; 
\item $F_D \ga = - \ga F_D$ and $\pi(a) \ga = \ga \pi(a)$ in the even case.
\end{enumerate}

For more information on $KK$-theory we refer the reader to the book by Blackadar, \cite{Bla:KOA}.
\medskip

The main difficulty is to prove the commutator condition $(3)$ and it is to this end that we have introduced and studied the modular transform in Section \ref{s:modtra}. To explain why this was necessary we first recall the notation
\[
S_\la := (\la \De^2/r + 1 + D^2)^{-1} : X \to X ,
\]
where $r \in (\| \De \|^2_\infty, \infty)$ is a fixed constant and $\la \geq 0$ is a variable. The next lemma presents the main algebraic reason for working with the modular resolvent $S_\la$ instead of the ordinary resolvent $R_\la = (\la + 1 + D^2)^{-1}$. Indeed, if the computation below is carried out with $R_\la$ instead of $S_\la$, then the commutator $[\De^2,T]$ has to be replaced by the commutator $[(1 + \la) \De^2,T]$ and there is then no gain in the decay properties when the variable $\la$ tends to infinity. This observation is responsible for the failure of the usual proof (\cite{BaJu:TBK}) of condition $(3)$.

We remark that it follows from Definition \ref{d:unbkas} that the conditions in Assumption \ref{modassu} are satisfied for the pair $(D,\De)$. %In fact all the results in this section except for Theorem \ref{t:kasmod} are valid under the conditions applied in Section \ref{s:modtra}.

\begin{lemma}\label{l:cruxalg}
Let $T \in \sL(X)$ be differentiable with respect to $(D,\De)$ (as in Definition \ref{d:unbkasII}). We have the identity
\[
S_\la \De^2 T - T \De^2 S_\la = S_\la [\De^2 , T] S_\la - (D S_\la)^* d_\De(T \De) S_\la - S_\la d_\De(\De T) D S_\la
\]
for all $\la \geq 0$.
\end{lemma}
\begin{proof}
Let first $\xi \in \sD(D^2)$ and notice that
\[
\begin{split}
& S_\la \De^2 T D^2(\xi) - (D^2 S_\la)^* T \De^2(\xi)  \\
& \q = (D S_\la)^*\De T \De D(\xi) - S_\la d_\De(\De T ) D(\xi) - (D^2 S_\la)^* T \De^2(\xi) \\
& \q = (D S_\la)^* D T \De^2(\xi) - (D S_\la)^* d_\De(T \De)(\xi)
- S_\la d_\De(\De T ) D(\xi) - (D^2 S_\la)^* T \De^2(\xi) \\
& \q = - (D S_\la)^* d_\De(T \De)(\xi) - S_\la d_\De(\De T ) D(\xi) .
\end{split}
\]
The result of the lemma then follows since
\[
\begin{split}
S_\la \De^2 T - T \De^2 S_\la
& = S_\la \De^2 T (D^2 + 1 + \la \De^2/r) S_\la \\ 
& \q - S_\la (1 + \la \De^2/r) T \De^2 S_\la
- (D^2 S_\la)^* T \De^2 S_\la \\
& = S_\la [\De^2, T] S_\la + S_\la \De^2 T D^2 S_\la - (D^2 S_\la)^* T \De^2 S_\la .\qedhere
\end{split}
\]
\end{proof}

In the next two lemmas we show that we may replace the bounded transform $F_D$ (up to a compact perturbation) by the modular transform $G_{D,\De}$ (in a slight disguise).

\begin{lemma}\label{l:fircomext}
Let $Z$ be an extra Hilbert $C^*$-module over $B$ and let $T : Z \to X$ be a bounded adjointable operator. Suppose that $(1 + D^2)^{-1} T \in \sK(Z,X)$ and that there exists a dense submodule $\sZ \su Z$ such that $T(\sZ) \su \sD(D)$. Then the unbounded operator
\[
\De^5 F_D \cd T - \frac{1}{\pi} D \cd \int_0^\infty (\la r)^{-1/2} \cd \De S_\la \De^5 \, d\la \cd T : \sZ \to X 
\]
is the restriction of an element in $\sK(Z,X)$.
\end{lemma}
\begin{proof}
It follows by Theorem \ref{t:comboutraII} that the difference
\[
\De^5 F_D T  - \frac{1}{\pi} \int_0^\infty (\la r)^{-1/2} \cd \De^6 S_\la \, d\la \cd D T : \sZ  \to X 
\]
is the restriction of an element in $\sK(Z,X)$.

Furthermore, we notice that the difference
\[
\De^6 S_\la D T - D \De S_\la \De^5 T : \sZ \to X 
\]
extends to a compact operator from $Z$ to $X$ for all $\la \geq 0$ (in fact both of the two terms have this property).

To prove the lemma, it therefore suffices to find a constant $C > 0$ such that
\[
\big\| \De^6 S_\la D \xi - D \De S_\la \De^5 \xi \big\| \leq C \cd (1 + \la)^{-3/4} \cd \| \xi \|
\]
for all $\la \geq 0$ and all $\xi \in \sD(D)$. Using Lemma \ref{l:algide} we see that
\[
\begin{split}
\De^6 S_\la D \xi - D \De S_\la \De^5 \xi
& = \De [\De^5, S_\la] D \xi - [D, \De S_\la \De^5] \xi \\
& = \De S_\la^{1/2} \Om_\la d(\De^5) S_\la^{1/2} \Om_\la \xi + \De S_\la d(\De^5) \Om_\la^* \Om_\la \xi \\
& \q - d(\De) S_\la \De^5 \xi + \la \cd \De S_\la d(\De^2/r) S_\la \De^5 \xi - \De S_\la d(\De^5) \xi
\end{split}
\]
for all $\la \geq 0$ and all $\xi \in \sD(D)$. The desired estimate then follows from Lemma \ref{l:eleestI} and Assumption \ref{modassu} $(2)$.
\end{proof}

\begin{lemma}\label{l:seccomext}
Let $Z$ be an extra Hilbert $C^*$-module over $B$ and let $T : Z \to X$ be a bounded adjointable operator. Suppose that $(1 + D^2)^{-1} T \in \sK(Z,X)$ and that there exists a dense submodule $\sZ \su Z$ such that $T(\sZ) \su \sD(D)$. Then the unbounded operator 
\[
F_D \De^5 T - \frac{1}{\pi} D \cd \int_0^\infty (\la r)^{-1/2} \De^4 S_\la \De^2 \, d\la \cd T : \sZ \to X
\]
is the restriction of an element in $\sK(Z,X)$.
\end{lemma}
\begin{proof}
We start by noting that
\[
[F_D,\De^5] T \in \sK(Z,X) .
\]
Indeed, this follows by using the integral formula
\[
F_D = D \cd \frac{1}{\pi} \int_0^\infty \la^{-1/2}(\la + 1 + D^2)^{-1} \, d\la
\]
and the fact that $[D,\De] : \sD(D) \to X$ has a bounded adjointable extension to $X$.

Now, by Lemma \ref{l:fircomext} we obtain that the difference of unbounded operators
\[
\De^5 F_D T - \frac{1}{\pi} D \cd \int_0^\infty (\la r)^{-1/2} \De S_\la \De^5 \, d\la \cd T : \sZ \to X
\]
is the restriction of an element in $\sK(Z,X)$.

We then remark that the difference
\[
D \De S_\la \De^5 T - D \De^4 S_\la \De^2 T : Z \to X
\]
is a compact operator (both of these terms are in fact compact).

To prove the lemma, it therefore suffices to find a constant $C > 0$ such that
\[
\big\| D \De [S_\la, \De^3] \De^2 \big\|_\infty \leq C \cd (1 + \la)^{-1} 
\]
for all $\la \geq 0$. However, using Lemma \ref{l:algide} we see that
\[
\begin{split}
D \De [S_\la, \De^3] \De^2 & = d(\De) [S_\la, \De^3] \De^2 + \De D [S_\la, \De^3] \De^2 \\
& = d(\De) [S_\la,\De^3] \De^2 - \De \big( \Om_\la^* \Om_\la d(\De^3) S_\la + D S_\la d(\De^3) D S_\la \big) \De^2 .
\end{split}
\]
The relevant estimate then follows by an application of Lemma \ref{l:eleome}, \ref{l:eleestI}, \ref{l:eleestII}, and \ref{l:eleestIV} in combination with Assumption \ref{modassu} $(2)$.
\end{proof}

\begin{prop}\label{p:comcom}
Let $T_0,T_1 \in \sL(X)$ and suppose that the following holds:
\begin{enumerate}
\item $T_0$ is differentiable with respect to $(D,\De)$ and $(1 + D^2)^{-1} T_0 \in \sK(X)$.
\item $(1 + D^2)^{-1} T_1 \in \sK(X)$ and $(T_1 \De)(\sD(D)) \su \sD(D)$.
\end{enumerate}
Then the bounded adjointable operator $[F_D, \De^5 T_0 \De^5] T_1 : X \to X$ is compact.
\end{prop}
\begin{proof}
By Lemma \ref{l:fircomext} and Lemma \ref{l:seccomext} it suffices to show that the difference
\[
\begin{split}
& \frac{1}{\pi} D \cd \int_0^\infty (\la r)^{-1/2} \De^4 S_\la \De^2  \, d\la \cd  T_0 \De^5 T_1 \\
& \q - \frac{1}{\pi} \De^5 T_0 D \cd \int_0^\infty  (\la r)^{-1/2} \De S_\la \De^5 \, d\la \cd T_1 : \De\big(\sD(D)\big) \to X
\end{split}
\]
extends to a compact operator on $X$. To this end, we notice that
\[
K_\la := D \De^4 S_\la \De^2 T_0 \De^5 T_1 - \De^5 T_0 D \De S_\la \De^5 T_1 : X \to X
\]
is compact for all $\la \geq 0$. In order to prove the proposition, it therefore suffices to find a constant $C > 0$ such that
\begin{equation}\label{eq:comcom}
\| K_\la \|_\infty \leq C \cd (1 + \la)^{-1}
\end{equation}
for all $\la \geq 0$. To show that this is indeed possible, we notice that 
\[
\begin{split}
& D \De^4 S_\la \De^2 T_0 \De^5 - \De^5 T_0 D \De S_\la \De^5 \\
& \q = D \De^4 ( S_\la \De^2 T_0 - T_0 \De^2 S_\la ) \De^5
+ (D \De^4 T_0 \De - \De^5 T_0 D) \De S_\la \De^5 .
\end{split}
\]
The relevant estimate for the second of these two terms then follows from Lemma \ref{l:eleestI} and the fact that $T_0$ is differentiable. To treat the first of these two terms we apply Lemma \ref{l:cruxalg} to see that
\[
\begin{split}
& D \De^4 ( S_\la \De^2 T_0 - T_0 \De^2 S_\la ) \De^5 \\
& \q = d(\De^4) ( S_\la \De^2 T_0 - T_0 \De^2 S_\la ) \De^5 \\
& \qq + \De^4 \Om_\la^* \big( S_\la^{1/2} [\De^2,T_0] S_\la -  \Om_\la d_\De(T_0 \De) S_\la - S_\la^{1/2} d_\De(\De T_0) \Om_\la^* S_\la^{1/2} \big) \De^5.
\end{split}
\]
The relevant estimate then follows by Lemma \ref{l:eleome}, \ref{l:eleestI}, \ref{l:eleestII}, and \ref{l:eleestIV}.
\end{proof}

\begin{thm}\label{t:kasmod}
Let $(X,D,\De)$ be an unbounded modular cycle from $\sA$ to the $C^*$-algebra $B$ (with grading operator $\ga : X \to X$ in the even case). Then the bounded transform $\big(X, D(1 + D^2)^{-1/2}\big)$ is a bounded Kasparov module from the $C^*$-algebra $A$ to the $C^*$-algebra $B$ of the same parity as $(X,D,\De)$ and with grading operator $\ga : X \to X$ in the even case.
\end{thm}
\begin{proof}
The only non-trivial issue is the compactness of the commutator $[F_D,\pi(a)] : X \to X$ for all $a \in A$. However, it already follows by Proposition \ref{p:comcom} that $(\De^5 + 1/n)^{-1}[F_D, \De^5 \pi(a) \De^5] (\De^5 + 1/n)^{-1}\pi(b) : X \to X$ is compact for all $a, b \in \sA$ and all $n \in \nn$. Moreover, it may be verified that
\[
\begin{split}
& [F_D, (\De^5 + 1/n)^{-1}] \De^5 \pi(a) \De^5 (\De^5 + 1/n)^{-1} \pi(b) \q \T{and} \\
& (\De^5 + 1/n)^{-1} \De^5 \pi(a) \De^5 [F_D, (\De^5 + 1/n)^{-1}] \pi(b)
\end{split}
\]
are compact operators on $X$. Using the density of $\sA$ in $A$ and the fact that $\De^5(1/n + \De^5)^{-1} \pi(a) \to \pi(a)$ in operator norm for all $a \in \sA$ we obtain that $[F_D, \pi(a)] \pi(b) \in \sK(X)$ for all $a,b \in A$. It then follows that $[F_D, \pi(a)] \in \sK(X)$ for all $a \in A$ by a standard trick in $KK$-theory.
\end{proof}

\begin{remark}
There is a much easier proof of Theorem \ref{t:kasmod} in the case where the unbounded modular cycle is \emph{Lipschitz regular} thus when the twisted commutator $|D| \pi(a) \De - \De \pi(a) |D| : \sD(D) \to X$ has a bounded extension for all $a \in \sA$. Indeed, it is then possible to follow \cite[Proposition 3.2]{CoMo:TST} more or less to the letter. It is however unclear whether the condition of Lipschitz regularity is compatible with the unbounded Kasparov product construction given in Section \ref{s:unbkaspro}. In fact, to our knowledge, this problem is not even well-understood in the case of the passage from $D$ to $\ov{g D g}$ (see Remark \ref{r:twispe} and \cite[Section 2.2]{CoMo:TST}). In this text, we have therefore chosen to avoid the extra Lipschitz regularity condition altogether.
\end{remark}

\section{Relation to the bounded Kasparov product}\label{s:boukaspro}
Throughout this section we let $\sA$ and $\sB$ be two $*$-algebras which satisfy the conditions in Assumption \ref{a:alg}. 

We consider an unbounded modular cycle $(Y,D,\Ga)$ from $\sB$ to an auxiliary $C^*$-algebra $C$. The parity of $(Y,D,\Ga)$ is denoted by $p \in \{0,1\}$. Furthermore, we let $X$ be a differentiable Hilbert $C^*$-module from $\sA$ to $\sB$ with differentiable generating sequence $\{ \xi_n\}_{n = 1}^\infty$. We finally suppose that the $*$-homomorphism $\pi_A : A \to \sL(X)$ factorizes through the compact operators $\sK(X) \su \sL(X)$.

As a consequence of Theorem \ref{t:unbkas} we obtain that the triple 
\[
\big( X \hot_B Y, (1 \ot D)_\De, \De \big)
\]
is an unbounded modular cycle from $\sA$ to $C$ of the same parity as $(Y,D,\Ga)$. Thus, by an application of Theorem \ref{t:kasmod} we have a bounded Kasparov module
\[
\Big( X \hot_B Y, (1 \ot D)_\De \cd \big( 1 + (1 \ot D)_\De^2 \big)^{-1/2} \Big)
\]
from $A$ to $C$ and hence a class $[ (1 \ot D)_\De]$ in the $KK$-group $KK_p(A,C)$.

On the other hand, since $\pi_A(a) \in \sK(X)$ for all $a \in A$, our differentiable Hilbert $C^*$-module $X$ defines an even bounded Kasparov module $\big( X, 0 \big)$ from $A$ to $B$, and hence a class $[X]$ in the even $KK$-group $KK_0(A,B)$. The grading operator is here just the identity operator on $X$. On top of this, we know from Theorem \ref{t:kasmod} that our original unbounded modular cycle $(Y,D,\Ga)$ yields a bounded Kasparov module
\[
(Y, D(1 + D^2)^{-1/2})
\]
from $B$ to $C$ and therefore we also have a class $[D]$ in the $KK$-group $KK_p(B,C)$.
\medskip

\emph{Under the condition that $A$ is separable and $B$ is $\si$-unital, we prove in this final section that the identity
\begin{equation}\label{eq:kaspro}
[ (1 \ot D)_\De] = [X] \hot_B [D] 
\end{equation}
holds inside the $KK$-group $KK_p(A,C)$, where
\[
\hot_B : KK_0(A,B) \ti KK_p(B,C) \to KK_p(A,C) 
\]
denotes the interior Kasparov product in $KK$-theory.}
\medskip

To ease the notation, we define
\[
\begin{split}
& F_\De := (1 \ot D)_\De \cd \big( 1 + (1 \ot D)_\De^2 \big)^{-1/2} \in \sL(X \hot_B Y) \q \T{and} \\
& F := D (1 + D^2)^{-1/2} \in \sL(Y) .
\end{split}
\]

For the rest of this section, we assume that the $C^*$-algebra $A$ is separable and that the $C^*$-algebra $B$ has a countable approximate identity (thus that $B$ is $\si$-unital).

\begin{remark}
Even though the interior Kasparov product in $KK$-theory is only constructed under the assumption that $A$ is separable and $B$ is $\si$-unital we do not rely on these assumptions for the construction of the unbounded Kasparov product. The bounded Kasparov module $(X \hot_B Y, F_\De)$ therefore exists regardless of these assumptions on the $C^*$-algebras $A$ and $B$.
\end{remark}

Due to a result of Connes and Skandalis we may focus on proving that $F_\De$ is an $F$-connection, \cite[Theorem A.3]{CoSk:LIF}. Or in other words, if we can show that
\begin{equation}\label{eq:fcon}
F T_{\xi}^* - T_\xi^* F_\De \in \sK(X \hot_B Y, Y)
\end{equation}
for all $\xi \in X$ we may conclude that the identity in Equation \eqref{eq:kaspro} holds. We recall here that
\[
T_\xi^* : X \hot_B Y \to Y \q T_\xi^* : x \ot_B y \mapsto \pi_B(\inn{\xi,x})(y)
\]
for all $x \in X$, $y \in Y$.

\begin{remark}
In the work of Kucerovsky, \cite[Theorem 13]{Kuc:PUM}, conditions are given for recognizing unbounded representatives for the interior Kasparov product. These conditions can not be applied in our setting since our unbounded cycles are not unbounded Kasparov modules in the sense of \cite{BaJu:TBK}. Indeed, the main difference is that we are considering a twisted commutator condition (see Definition \ref{d:unbkas}) instead of the straight commutator condition applied in \cite{BaJu:TBK}. 
\end{remark}

We start by replacing the connection condition in Equation \eqref{eq:fcon} by something more manageable. Let us recall that $\Phi : X \hot_B Y \to \ell^2(\nn,Y)$ is defined by $\Phi : x \ot_B y \mapsto \sum_{n = 1}^\infty \pi_B(\inn{\xi_n,x})(y) \de_n$ for all $x \in X$, $y \in Y$. Furthermore, we have the modular operator $\De := \Phi^* (1 \ot \Ga) \Phi : X \hot_B Y \to X \hot_B Y$.

\begin{lemma}\label{l:comreskas}
\[
\big( i + (1 \ot D)_\De \big)^{-1} \De \in \sK(X \hot_B Y) .
\]
\end{lemma}
\begin{proof}
This follows by Proposition \ref{p:compacres} since $\Phi \Phi^* \in \T{Im}\big( K(\pi_B) : K(B) \to \sL(\ell^2(\nn,Y)) \big)$, see also Proposition \ref{p:stamodcyc}.
\end{proof}

\begin{lemma}\label{l:repconcon}
Suppose that there exists a $k \in \nn$ such that
\[
(1 \ot F \Ga^k) \Phi \De^k - (1 \ot \Ga^k) \Phi \De^k F_\De \in \sK(X \hot_B Y, \ell^2(\nn,Y))
\]
then we have that
\[
F T_{\xi}^* - T_\xi^* F_\De \in \sK(X \hot_B Y, Y)
\]
for all $\xi \in X$.
\end{lemma}
\begin{proof}
We first show that 
\begin{equation}\label{eq:firinccon}
(1 \ot F \Ga^k) \Phi - (1 \ot \Ga^k) \Phi F_\De \in \sK(X \hot_B Y, \ell^2(\nn,Y)) .
\end{equation}
To this end, we notice that
\[
\begin{split}
& (1 \ot F \Ga^k) \Phi \De^k (\De^k + 1/n)^{-1} - (1 \ot \Ga^k) \Phi \De^k (\De^k + 1/n)^{-1} F_\De \\
& \q = (1 \ot F \Ga^k) \Phi \De^k (\De^k + 1/n)^{-1} - (1 \ot \Ga^k) \Phi \De^k F_\De (\De^k + 1/n)^{-1} \\
& \qqq - (1 \ot \Ga^k) \Phi \De^k (\De^k + 1/n)^{-1} [F_\De, \De^k] (\De^k + 1/n)^{-1} \\ 
& \qqq \qqq \in \sK(X \hot_B Y, \ell^2(\nn,Y))
\end{split}
\]
for all $n \in \nn$. Indeed, this is a consequence of the assumptions of the present lemma and the fact that $\De [F_\De,\De] : X \hot_B Y \to X \hot_B Y$ is compact (this last assertion follows by Lemma \ref{l:comreskas}). The inclusion in Equation \eqref{eq:firinccon} now holds since the sequence $\big\{ (1 \ot \Ga)^{1/2} \Phi \De^k (\De^k + 1/n)^{-1} \big\}_{n = 1}^\infty$ converges to $(1 \ot \Ga)^{1/2} \Phi : X \hot_B Y \to \ell^2(\nn,Y)$ in operator norm.

Our next step is to show that
\begin{equation}\label{eq:secinccon}
\Phi  F_\De - (1 \ot F) \Phi \in \sK(X \hot_B Y, \ell^2(\nn,Y)) .
\end{equation}
In this respect, we remark that
\[
\begin{split}
& \big(1 \ot F \Ga^k(\Ga^k + 1/n)^{-1} \big) \Phi - \big(1 \ot \Ga^k(\Ga^k + 1/n)^{-1} \big) \Phi F_\De \\
& \q = \big( 1 \ot (\Ga^k + 1/n)^{-1} F \Ga^k \big) \Phi - \big(1 \ot \Ga^k(\Ga^k + 1/n)^{-1} \big) \Phi F_\De \\
& \qqq - \big( 1 \ot (\Ga^k + 1/n)^{-1}[F,\Ga^k] \Ga^k (\Ga^k + 1/n)^{-1} \big) \Phi \in \sK(X \hot_B Y, \ell^2(\nn,Y))
\end{split}
\]
for all $n \in \nn$. Indeed, this is a consequence of the inclusion in Equation \eqref{eq:firinccon} and the fact that $\big( 1 \ot [F,\Ga^k] \Ga^k (\Ga^k + 1/n)^{-1} \big) \Phi \in \sK(X \hot_B Y, \ell^2(\nn,Y))$ (recall that $\Phi \Phi^*$ lies in the image of the $*$-homomorphism $K(\pi_B) : K(B) \to \sL( \ell^2(\nn,Y))$). The inclusion in Equation \eqref{eq:secinccon} now follows since the sequence $\big\{ \big(1 \ot \Ga^k (\Ga^k + 1/n)^{-1} \big) \Phi \big\}_{n = 1}^\infty$ converges to $\Phi : X \hot_B Y \to \ell^2(\nn,Y)$ in operator norm.

By the definition of $\Phi : X \hot_B Y \to \ell^2(\nn,Y)$ we see from Equation \eqref{eq:secinccon} that
\begin{equation}\label{eq:fouinccon}
T_{\xi_n}^* F_\De - F T_{\xi_n}^* \in \sK(X \hot_B Y, Y)
\end{equation}
for all $n \in \nn$. Let now $b \in B$ and $n \in \nn$ be given. We then have that
\[
\begin{split}
T_{\xi_n \cd b}^* F_\De - F T_{\xi_n \cd b}^* 
& = \pi_B(b^*) T_{\xi_n}^* F_\De - F \pi_B(b^*) T_{\xi_n}^* \\
& = \pi_B(b^*) \big( T_{\xi_n}^* F_\De - F T_{\xi_n}^* \big) - [F, \pi_B(b^*)] T_{\xi_n}^* .
\end{split}
\]
Thus, since $(Y,F)$ is a bounded Kasparov module we deduce from Equation \eqref{eq:fouinccon} that
\begin{equation}\label{eq:fivinccon}
T_{\xi_n \cd b}^* F_\De - F T_{\xi_n \cd b}^* \in \sK(X \hot_B Y, Y) .
\end{equation}
Since the sequence $\{ \xi_n \}_{n = 1}^\infty$ generates $X$ as a Hilbert $C^*$-module over $B$, we conclude from Equation \eqref{eq:fivinccon} that
\[
T_{\xi}^* F_\De - F T_{\xi}^* \in \sK(X \hot_B Y, Y)
\]
for all $\xi \in X$. This proves the lemma.
\end{proof}

Let us apply the notation
\[
S_\la := \big(\la \De^2/r + 1 + (1 \ot  D)_\De^2\big)^{-1} \q \T{and} \q
T_\la := \big(\la (1 \ot \Ga)^2/r + 1 + (1 \ot D)^2\big)^{-1} ,
\]
where $r \in ( \| \De\|^2_\infty + \| \Ga \|^2_\infty,\infty)$ is a fixed constant and $\la \geq 0$ is a variable.

The next lemma relates these two modular resolvents to one another:
%We will in the following often put
%\[
%D := 1 \ot D : \sD(1 \ot D) \to \ell^2(\nn,Y) \q \T{and} \q \Ga := 1 \ot \Ga : \ell^2(\nn,Y) \to \ell^2(\nn,Y)
%\]

\begin{lemma}\label{l:estboukas}
The difference
\begin{equation}\label{eq:difcomkas}
(1 \ot \Ga^5) \Phi \De^2 S_\la \De^4 (1 \ot D)_\De - (1 \ot D\Ga^4) T_\la (1 \ot \Ga^2) \Phi \De^5
: \sD\big( (1 \ot D) \Phi \big) \to \ell^2(\nn,Y)
\end{equation}
extends to a compact operator $K_\la : X \hot_B Y \to \ell^2(\nn,Y)$ and there exists a constant $C > 0$ such that
\begin{equation}\label{eq:estboukas}
\| K_\la \|_\infty \leq C \cd (1 + \la)^{-3/4}
\end{equation}
for all $\la \geq 0$.
\end{lemma}
\begin{proof}
It is not hard to see that the difference in Equation \eqref{eq:difcomkas} has a compact extension $K_\la : X \hot_B Y \to \ell^2(\nn,Y)$ for all $\la \geq 0$ (in fact this holds for each of the two terms). We may thus focus our attention on providing the operator norm estimate in Equation \eqref{eq:estboukas}

Our first step in this direction is to notice that it is enough to consider the difference
\[
(1 \ot \Ga^5) \Phi \De^2 S_\la (1 \ot D)_\De \De^4 - (1 \ot \Ga^5 D) T_\la (1 \ot \Ga) \Phi \De^5 : \sD( (1 \ot D) \Phi) \to \ell^2(\nn,Y)
\]
of unbounded operators. This follows since we may dominate the operator norm (uniformly in $\la \geq 0$) of each of the bounded adjointable operators
\[
\begin{split}
& (1 \ot \Ga^5) \Phi \De^2 S_\la d_\De(\De^3) \, \, , \, \, \, (1 \ot D \Ga^4) [T_\la , 1 \ot \Ga] (1 \ot \Ga) \Phi \De^5 \q \T{and} \\
& d_{1 \ot \Ga}(1 \ot \Ga^4) T_\la (1 \ot \Ga^2) \Phi \De^5 : X \hot_B Y \to \ell^2(\nn,Y)
\end{split}
\]
by $C_0 \cd (1 + \la)^{-3/4}$ for some constant $C_0 > 0$. To see that this is indeed the case, it suffices to apply Lemma \ref{l:eleome}-Lemma \ref{l:eleestII}. %the Appendix (Subsection \ref{ss:preopenor}).

Our next step is to define the unbounded operator
\[
\begin{split}
M_\la & := (1 \ot \Ga^3) T_\la \Big(\Phi \De^2 - (1 \ot \Ga^2) \Phi + d_{1 \ot \Ga}\big((1 \ot \Ga) G\big) \Phi  (1 \ot D)_\De \Big) S_\la  (1 \ot D)_\De \\
& \q + (1 \ot \Ga^3) ((1 \ot D) T_\la)^* \Big( d_{1 \ot \Ga}(G) G (1 \ot \Ga)  \\ 
& \qqq \qq \qq + (1 \ot \Ga) G d_{1 \ot \Ga}(G) \Big) \Phi S_\la  (1 \ot D)_\De \\
& \q : \sD(  (1 \ot D)_\De) \to \ell^2(\nn,Y),
\end{split}
\]
where we recall the notation $G := \Phi \Phi^* : \ell^2(\nn,Y) \to \ell^2(\nn,Y)$. It then follows by Lemma \ref{l:eleome}, Lemma \ref{l:eleestI} and Lemma \ref{l:eleestII} that there exists a constant $C_1 > 0$ such that
\begin{equation}\label{eq:estemm}
\| M_\la (\xi) \| \leq C_1 (1 + \la)^{-3/4} \cd \| \xi \|
\end{equation}
for all $\la \geq 0$ and all $\xi \in \sD((1 \ot D)_\De)  \su X \hot_B Y$. Furthermore, by Proposition \ref{p:cruxideII} we have that
\[
(1 \ot \Ga^3) \big( \Phi \De^2 S_\la - T_\la (1 \ot \Ga^2) \Phi \big)  (1 \ot D)_\De = M_\la
\]
for all $\la \geq 0$. In order to provide the relevant estimate on $K_\la : X \hot_B Y \to \ell^2(\nn,Y)$ it therefore suffices to analyze the difference
\[
(1 \ot \Ga^5) T_\la (1 \ot \Ga^2) \Phi (1 \ot D)_\De \De^4 - \Ga^5 (1 \ot D) T_\la (1 \ot \Ga) \Phi \De^5 : \sD\big(  (1 \ot D) \Phi \big) \to \ell^2(\nn,Y)
\]
of unbounded operators.

However, we have that
\[
\begin{split}
& T_\la (1 \ot \Ga^2) \Phi  (1 \ot D)_\De(\xi) - (1 \ot D) T_\la (1 \ot \Ga) \Phi \De(\xi) \\
& \q = -  d(T_\la (1 \ot \Ga))  \Phi \De(\xi) - T_\la (1 \ot \Ga) d_{1 \ot \Ga}( G) \Phi(\xi)
\end{split}
\]
for all $\xi \in \sD( (1 \ot D) \Phi)$ and the result of the lemma therefore follows by one more application of the operator norm estimates in Lemma \ref{l:eleestI} and Lemma \ref{l:eleestII}.
\end{proof}

\begin{lemma}\label{l:difboukas}
The unbounded operator
\[
\begin{split}
& \int_0^\infty (\la r)^{-1/2} \cd (1 \ot \Ga^5) \Phi \De^2 S_\la \De^4 \, d\la \cd (1 \ot D)_\De \\
& \q - (1 \ot D) \cd \int_0^\infty (\la r)^{-1/2} \cd (1 \ot \Ga^4)  T_\la (1 \ot \Ga^2) \, d\la \cd \Phi \De^5 \\
& \qqq : \sD( (1 \ot D) \Phi) \to \ell^2(\nn,Y)
\end{split}
\]
is the restriction of an operator in  $\sK( X \hot_B Y, \ell^2(\nn,Y))$.
\end{lemma}
\begin{proof}
This follows in a straightforward way by an application of Lemma \ref{l:estboukas}.
\end{proof}

We are now ready to prove our final main theorem:

\begin{thm}
Suppose that $X$ is a differentiable Hilbert $C^*$-module from $\sA$ to $\sB$ with left action $A \to \sL(X)$ factorizing through the compacts, $\sK(X)$. Suppose moreover that $(Y,D,\Ga)$ is an unbounded modular cycle from $\sB$ to $C$. Then the bounded adjointable operator $F_\De : X \hot_B Y \to X \hot_B Y$ is an $F$-connection. In particular, we have the identity
\[
[(1 \ot D)_\De] = [X] \hot_B [D]
\]
inside the $KK$-group $KK_p(A,C)$, when $A$ is separable and $B$ is $\si$-unital. 
\end{thm}
\begin{proof}
By Lemma \ref{l:repconcon}, we only need to show that
\[
(1 \ot F \Ga^5) \Phi \De^5 - (1 \ot \Ga^5) \Phi \De^5 F_\De \in \sK(X \hot_B Y, \ell^2(\nn,Y)) .
\]
However, by Lemma \ref{l:seccomext} (and a version of this lemma obtained by taking adjoints), we may just check that the difference
\[
\begin{split}
& (1 \ot D)  \cd \int_0^\infty (\la r)^{-1/2} (1 \ot \Ga^4) T_\la (1 \ot \Ga^2) \, d\la \cd \Phi \De^5 \\
& \qqq - \int_0^\infty (\la r)^{-1/2} (1 \ot \Ga^5) \Phi \De^2 S_\la \De^4 \, d\la \cd (1 \ot D)_\De \\
& \qqq \qqq  : \sD\big( (1 \ot D) \Phi \big)  \to \ell^2(\nn,Y)
\end{split}
\]
is the restriction of an element in $\sK(X \hot_B Y, \ell^2(\nn,Y))$. But this is a consequence of Lemma \ref{l:difboukas}.
\end{proof}

\section{Appendix: Norm estimates of error terms}
In this appendix we have collected various operator norm estimates needed in the treatment of the modular transform (Section \ref{s:modtra}) and for the comparison result between the unbounded Kasparov product and the bounded Kasparov product (Section \ref{s:boukaspro}).

The general setting will be exactly as in Section \ref{s:modtra} and the conditions in Assumption \ref{modassu} will in particular be in effect. We recall the notation for a few bounded adjointable operators acting on the Hilbert $C^*$-module $Y$:
\[
\begin{split}
& S_\la := (\la \De^2/r + 1 + D^2)^{-1} \, \, , \, \, \, R_\la := (\la + 1 + D^2)^{-1} \q \T{and} \\
& K := 1 - \De^2/r \, \, , \, \, \, X_\la := \la \cd R_\la K ,
\end{split}
\]
where $r \in ( \| \De \|^2_\infty , \infty)$ is a fixed constant and $\la \geq 0$ is variable.

\subsection{Preliminary operator norm estimates}\label{ss:preopenor}
We start with a string of elementary operator norm estimates which will be needed throughout this appendix and in many places in the main text as well.

\begin{lemma}\label{l:eleome} 
The unbounded operator $S_\la^{1/2} D : \sD(D) \to Y$ has a bounded adjointable extension $\Om_\la : Y \to Y$ and we have the operator norm estimate
\[
\| \Om_\la \|_\infty \leq 1 \q \mbox{for all } \la \geq 0 .
\]
\end{lemma}
\begin{proof}
Let $\la \geq 0$ be given. Consider the unbounded operator $E_\la : \T{Im}(S_\la^{1/2}) \to Y$ defined by $E_\la : S_\la^{1/2} \xi \mapsto D S_\la \xi$. It is then clear that $S_\la^{1/2} D \su E_\la^*$. Furthermore, for each $\xi \in Y$ we have that
\[
\inn{E_\la S_\la^{1/2} \xi, E_\la S_\la^{1/2} \xi} = \inn{S_\la D^2 S_\la \xi, \xi} \leq \binn{S_\la (\la \De^2/r + 1 + D^2) S_\la \xi,\xi}
= \inn{S_\la^{1/2} \xi, S_\la^{1/2} \xi} .
\]
It therefore follows that $E_\la : \T{Im}(S_\la^{1/2}) \to Y$ has a bounded extension to $Y$, $\ov{E_\la} : Y \to Y$ and furthermore that
$\big\| \ov{E_\la} \big\|_\infty \leq 1$. This shows that $E_\la^*$ is everywhere defined and that $E_\la^* = (\ov{E_\la})^*$. We may then conclude that $\ov{S_\la^{1/2} D} = E_\la^*$ and that $\big\| \ov{S_\la^{1/2} D} \big\|_\infty \leq 1$. This proves the lemma.
\end{proof}

\begin{lemma}\label{l:algide}
Let $\la \geq 0$ be given. We have the identities
\[
\begin{split}
& D(\De S_\la - S_\la \De) = \Om_\la^* \Om_\la d(\De) S_\la + D S_\la d(\De) D S_\la \q \mbox{and} \\
& \De S_\la - S_\la \De = S_\la^{1/2} \Om_\la d(\De) S_\la + S_\la d(\De) D S_\la .
\end{split}
\]
\end{lemma}
\begin{proof}
We only prove the first of these two identities. The second identity can be proved by a similar but easier argument.

Using that $\sD(D^2) \su Y$ is a core for $D : \sD(D) \to Y$ it follows that
\begin{equation}\label{eq:omegadee}
\Om_\la^* \Om_\la D = D - D S_\la (\la \De^2/r + 1)
\end{equation}
on the common domain $\sD(D) \su Y$. We then obtain our identity from the computation:
\[
\begin{split}
\Om_\la^* \Om_\la d(\De) S_\la + D S_\la d(\De) D S_\la
& = D \De S_\la - D S_\la (\la \De^2/r + 1) \De S_\la - D S_\la \De D^2 S_\la  \\
& = D \De S_\la - D S_\la \De . \qedhere
\end{split}
\]
\end{proof}

\begin{lemma}\label{l:eleestI}
Let $\la \geq 0$ be given. We have the operator norm estimate 
\[
\| \De S_\la^{1/2} \|_\infty \leq \frac{\sqrt{r}}{\sqrt{1 + \la}} .
\]
\end{lemma}
\begin{proof}
This follows by noting that
\[
0 \leq S_\la^{1/2} (\la + 1)(\De^2 /r ) S_\la^{1/2} \leq 1 . \qedhere
\]
\end{proof}

\begin{lemma}\label{l:eleestII} 
There exists a constant $C > 0$ such that
\[
\| \De^2 S_\la \De^{1/2} \|_\infty \leq C \cd (1 + \la)^{-1}
\]
for all $\la \geq 0$.
\end{lemma}
\begin{proof}
Using Lemma \ref{l:algide} we obtain that
\[
\De^2 S_\la \De^{1/2} = \De S_\la \De^{3/2} + \De \cd S_\la^{1/2} \Om_\la d(\De) S_\la \cd \De^{1/2} + \De \cd S_\la d(\De) D S_\la \cd \De^{1/2}
\]
The desired estimate now follows by Lemma \ref{l:eleome}, Lemma \ref{l:eleestI}, and the standing Assumption \ref{modassu}. Recall here that the sequence $\{ \De^{1/2}(1/n + \De)^{-1/2}\}_{n = 1}^\infty$ converges strictly to the identity.
\end{proof}

\begin{lemma}\label{l:eleestIV}
There exists a constant $C > 0$ such that
\[
\| \De D S_\la \De^{1/2} \|_\infty \leq C \cd (1 + \la)^{-1/2}
\]
for all $\la \geq 0$.
\end{lemma}
\begin{proof}
By an application of Lemma \ref{l:algide}, we may compute as follows:
\[
\begin{split}
& \De D S_\la \De^{1/2} = D \De S_\la \De^{1/2} - d(\De) S_\la \De^{1/2} \\
& \q = D S_\la \De^{3/2} + \Om_\la^* \Om_\la d(\De) S_\la \cd \De^{1/2} + D S_\la d(\De) D S_\la \cd \De^{1/2}
- d(\De) S_\la \De^{1/2} .
\end{split}
\]
The relevant estimate is now a consequence of Lemma \ref{l:eleome}, Lemma \ref{l:eleestI}, and Assumption \ref{modassu}.
\end{proof}

\begin{lemma}\label{l:eleestV}
Let $m \in [2,\infty)$ be given. There exists a constant $C > 0$ such that
\[
\| D S_\la \De^m (i + D)^{-1} \|_\infty \leq C \cd (1 + \la)^{-1}
\]
for all $\la \geq 0$.
\end{lemma}
\begin{proof}
Using Equation \eqref{eq:omegadee}, we compute as follows:
\[
\begin{split}
& D S_\la \la (\De^m/r) (i + D)^{-1} \\
& \q = D \De^{m-2} (i + D)^{-1} - D S_\la \De^{m-2} (i + D)^{-1}
- \Om_\la^* \Om_\la D \De^{m-2} (i + D)^{-1} .
\end{split}
\]
Since $D \De^{m-2} (i + D)^{-1} : Y \to Y$ is a bounded adjointable operator, we obtain the relevant estimate by Lemma \ref{l:eleome}.
\end{proof}

\begin{lemma}\label{l:eleestVI}
Let $m \geq 3$ be given. There exists a constant $C > 0$ such that
\[
\begin{split}
& \| S_\la^{3/2} \De^m (i + D)^{-1} \|_\infty \leq C \cd (1 + \la)^{-1 - 1/8}  \q \mbox{and} \\
& \| S_\la^{1/2} (1 - X_\la^*)^{-1} \De^m (i + D)^{-1} \|_\infty \leq C \cd (1 + \la)^{-1/8}
\end{split}
\] 
for all $\la \geq 0$.
\end{lemma}
\begin{proof}
To prove the first of the two estimates we apply Lemma \ref{l:algide} to obtain that
\[
\begin{split}
& S_\la^{3/2} \De^m (i + D)^{-1} \\
& \q  = S_\la^{1/2} \De S_\la \cd \De^{m-1} (i + D)^{-1}
- S_\la \Om_\la d(\De) S_\la \cd \De^{m-1} (i + D)^{-1} \\
& \qqq - S_\la^{3/2} d(\De) D S_\la \cd \De^{m-1} (i + D)^{-1} .
\end{split}
\]
After consulting Lemma \ref{l:eleestI}, Lemma \ref{l:eleestII}, and Lemma \ref{l:eleestV}, we see that it suffices to find a constant $C_1 > 0$ such that
\[
\| S_\la^{1/2} \Om_\la \De^{1/2} \|_\infty \leq C_1 \cd (1 + \la)^{-1/8} .
\]
But this follows by noting that $\| S_\la^{1/2} \Om_\la \De^{1/2} \|^2_\infty = \| S_\la D \De D S_\la \|_\infty$. In fact, it follows from the proof of Lemma \ref{l:eleestIV} that $\| \De D S_\la \|_\infty \cd C_2 \cd (1 + \la)^{-1/4}$ for some constant $C_2 > 0$, which is independent of $\la \geq 0$. 

In order to prove the second of the two estimates, we remark that Lemma \ref{l:firreside} implies that
\[
(1 - X_\la^*)^{-1} = \sum_{n = 0}^\infty (X_\la^*)^n 
= 1 + \la K R_\la \sum_{n = 0}^\infty (X_\la^*)^n 
= 1 + \la K S_\la .
\]
The result then follows from the first estimate, which we already proved above.
\end{proof}

\subsection{Norm estimates of limit error terms}
Let us recall (from Subsection \ref{ss:prealgide}) that
\[
L_\la = I(S_\la K \la \De^3) R_\la : Y \to Y
\]
for all $\la \geq 0$ (where $I(\cd) = [D^2, \cd]$).

\begin{prop}\label{p:esterr}
There exists a constant $C > 0$ such that
\begin{equation}\label{eq:esterr}
\| \De^2 \cd L_\la \|_\infty \leq C \cd (1 + \la)^{-1/2}
\end{equation}
for all $\la \geq 0$.
\end{prop}
\begin{proof}
For each $\la \geq 0$ we rewrite $L_\la : Y \to Y$ in the following way:
\begin{equation}\label{eq:esterrII}
\begin{split}
L_\la 
& = D d(S_\la K \la \De^3) R_\la + d(S_\la K \la \De^3) D R_\la \\
& = - D S_\la d(\De^2/r) \la S_\la \cd K \la \De^3 R_\la + D S_\la \cd d(\De^3 - \De^5/r) \la R_\la \\
& \qqq - S_\la d(\De^2/r) \la S_\la \cd K \la \De^3 D R_\la
+ S_\la d(\De^3 - \De^5/r) \la \cd D R_\la .
\end{split}
\end{equation}
It is then not hard to see that the desired estimate follows from Lemma \ref{l:eleestII} and Lemma \ref{l:eleestIV}. 
\end{proof}

\subsection{Operator norm estimates of truncated error terms}\label{ss:truerr}
Let us recall (again from Subsection \ref{ss:prealgide}) that
\[
L_\la(m) = I\big( (1 - X_\la^m) S_\la K \la \De^3 \big) R_\la  : Y \to Y
\]
for all $\la \geq 0$ and all $m \in \nn$.

\begin{prop}\label{p:truerr}
There exists a constant $C > 0$ such that
\[
\big\| \De^2 L_\la(m) (i + D)^{-1} \big\|_\infty \leq C \cd (1 + \la)^{-1/8}
\]
for all $\la \geq 0$ and all $m \in \nn$.
\end{prop}
\begin{proof}
We first notice that
\begin{equation}\label{eq:truerr}
L_\la(m) = (1 - X_\la^m) L_\la - I(X_\la^m) S_\la K \la \De^3 R_\la
\end{equation}
for all $\la \geq 0$ and all $m \in \nn$. We now estimate these two terms separately.

We begin with the easiest one: $(1 - X_\la^m) L_\la : Y \to Y$. Using the identity in Equation \eqref{eq:esterrII} we obtain that
\[
\begin{split}
& L_\la \cd (i + D)^{-1} \\
& \q = - D S_\la d(\De^2/r) \la^2 S_\la K \De^3 R_\la (i + D)^{-1} + D S_\la d(\De^3 - \De^5/r) \la R_\la (i + D)^{-1} \\
& \qq - S_\la d(\De^2/r) \la^2 S_\la K \De^3 D R_\la (i + D)^{-1} + S_\la d(\De^3 - \De^5/r) \la D R_\la (i + D)^{-1} .
\end{split}
\]
It then follows by Lemma \ref{l:eleome}, Lemma \ref{l:eleestI}, and Lemma \ref{l:eleestII} that there exists a constant $C_0 > 0$ such that
\[
\| L_\la \cd (i + D)^{-1} \|_\infty \leq C_0 \cd (1 + \la)^{-1/4}
\]
for all $\la \geq 0$. Since $\| 1 - X_\la^m \|_\infty \leq 2$ for all $\la \geq 0$ and all $m \in \nn$ we obtain the relevant estimate for the first term in Equation \eqref{eq:truerr}.

To take care of the second term in Equation \eqref{eq:truerr} we let $l \geq 3$ be given. It then suffices to estimate the operator norm of the bounded adjointable operator
\[
\De^2 I(X_\la^m) S_\la \De^l (i + D)^{-1} : Y \to Y
\]
uniformly in $m \in \nn$ and $\la \geq 0$. In order to achieve this goal we notice that
\[
\begin{split}
I(X_\la^m) S_\la
& = \sum_{j = 0}^{m-1} X_\la^j \big( D d(X_\la) + d(X_\la) D \big) X_\la^{m-1-j} S_\la \\
& = - \la \cd \sum_{j = 0}^{m-1} X_\la^j \big( D R_\la d(\De^2/r) + R_\la d(\De^2/r) D \big) X_\la^{m-1-j} S_\la .
\end{split}
\]
We now define
\[
\begin{split}
& A_\la(m) := \sum_{j = 0}^{m-1} \De^2 X_\la^j \cd D R_\la d(\De^2) \cd X_\la^{m-1 - j} S_\la \De^l (i + D)^{-1} \q \T{and} \\
& B_\la(m) := \sum_{j = 0}^{m-1} \De^2 X_\la^j \cd R_\la d(\De^2) D \cd X_\la^{m-1 - j} S_\la \De^l (i + D)^{-1}
\end{split}
\]
and it follows that $\De^2 I(X_\la^m) S_\la \De^l (i + D)^{-1} = -(\la/r) \cd  A_\la(m) - (\la /r) \cd B_\la(m) : Y \to Y$ for all $\la \geq 0$ and all $m \in \nn$.

Our next step is to estimate the operator norm of each of the terms $A_\la(m) : Y \to Y$ and $B_\la(m) : Y \to Y$ uniformly in $\la \geq 0$ and $m \in \nn$. We start with $A_\la(m)$. Using the Cauchy-Schwarz inequality together with Lemma \ref{l:firreside}, we obtain that
\[
\begin{split}
\| A_\la(m) \|^2_\infty & \leq \big\| \De^2 \sum_{j = 0}^{m-1} X_\la^j D R_\la D R_\la (X_\la^*)^j \De^2 \big\|_\infty \\
& \q \cd \big\| (-i + D)^{-1} \De^l S_\la \cd \sum_{j = 0}^{m-1} (X_\la^*)^j d(\De^2) d(\De^2) X_\la^j \cd S_\la \De^l (i + D)^{-1} \big\|_\infty \\
& \leq \| \De^2 S_\la \De^2 \|_\infty \cd \| d(\De^2) \|^2_\infty \cd \big\| (-i + D)^{-1} \De^l (1 - X_\la)^{-1} \\
& \qqq \cd \sum_{j = 0}^{m-1} X_\la^j R_\la^2 (X_\la^*)^j \cd (1 - X_\la^*)^{-1} \De^l (i + D)^{-1} \big\|_\infty \\
& \leq \| \De^2 S_\la \De^2 \|_\infty \cd \| d(\De^2) \|^2_\infty \cd (1 + \la)^{-1} \\
& \q \cd \big\| (-i + D)^{-1} \De^l (1 - X_\la)^{-1} S_\la (1 - X_\la^*)^{-1} \De^l (i + D)^{-1} \big\|_\infty .
\end{split}
\]
It then follows by Lemma \ref{l:eleestI} and Lemma \ref{l:eleestVI} that there exists a constant $C_1 > 0$ such that $\| A_\la(m) \|_\infty \leq C_1 \cd (1 + \la)^{-1 - 1/8}$ for all $m \in \nn$ and all $\la \geq 0$.

We continue with $B_\la(m)$. Another application of the Cauchy-Schwarz inequality and Lemma \ref{l:firreside} yields that
\[
\begin{split}
\| B_\la(m) \|^2_\infty
& \leq \big\| \De^2 \sum_{j = 0}^{m-1} X_\la^j R_\la d(\De^2) d(\De^2) R_\la (X_\la^*)^j \De^2 \big\|_\infty \\
& \q \cd \big\| (-i + D)^{-1} \De^l S_\la \cd \sum_{j = 0}^{m-1} (X_\la^*)^j D^2 X_\la^j \cd S_\la \De^l (i + D)^{-1} \big\|_\infty \\
& \leq (1 + \la)^{-1} \cd \| d(\De^2) \|^2_\infty \cd \| \De^2 S_\la \De^2 \|_\infty \cd \big\| (-i + D)^{-1} \De^l (1 - X_\la)^{-1} \\ 
& \qqq \cd \sum_{j = 0}^{m-1} X_\la^j R_\la D^2 R_\la (X_\la^*)^j
\cd (1 - X_\la^*)^{-1} \De^l (i + D)^{-1} \big\|_\infty \\
& \leq (1 + \la)^{-1} \cd \| d(\De^2) \|^2_\infty \cd \| \De^2 S_\la \De^2 \|_\infty \\
& \q \cd \big\| (-i + D)^{-1} \De^l (1 - X_\la)^{-1} S_\la (1 - X_\la^*)^{-1} \De^l (i + D)^{-1} \big\|_\infty .
\end{split}
\]
As a consequence of Lemma \ref{l:eleestI} and Lemma \ref{l:eleestVI} we may then find a constant $C_2 > 0$ such that $\| B_\la(m) \|_\infty \leq C_2 (1 + \la)^{-1 -1/8}$ for all $m \in \nn$ and all $\la \geq 0$. Combining our estimates we find that
\[
\big\| \De^2 I(X_\la^m) S_\la \De^l (i + D)^{-1} \big\|_\infty \leq (C_2/r + C_3/r) \cd (1 + \la)^{-1/8}
\]
for all $m \in \nn$ and all $\la \geq 0$. This ends the proof of the proposition.
\end{proof}

\bibliographystyle{amsalpha-lmp}

\providecommand{\bysame}{\leavevmode\hbox to3em{\hrulefill}\thinspace}
\providecommand{\MR}{\relax\ifhmode\unskip\space\fi MR }
% \MRhref is called by the amsart/book/proc definition of \MR.
\providecommand{\MRhref}[2]{%
  \href{http://www.ams.org/mathscinet-getitem?mr=#1}{#2}
}
\providecommand{\href}[2]{#2}

\end{document}